\documentclass[11pt]{article} 

\usepackage{geometry}                		
\usepackage{graphicx}				
\usepackage{wrapfig}
\usepackage{amsmath,amssymb,amsthm}							
\usepackage{multirow}
\usepackage{tabularx,epsfig,color,caption}
\usepackage{algorithm,algorithmic,mathtools,bm}
\usepackage{epstopdf}

\newtheorem{theorem}{THEOREM}[section]

\newtheorem{lemma}{LEMMA}[section]

\newcommand{\be}{\begin{equation}}
\newcommand{\ee}{\end{equation}}
\newcommand{\ba}{\begin{array}}
\newcommand{\ea}{\end{array}}
\newcommand{\bea}{\begin{eqnarray}}
\newcommand{\eea}{\end{eqnarray}}
\newcommand{\beas}{\begin{eqnarray*}}
\newcommand{\eeas}{\end{eqnarray*}}

\title{Analysis and computation of some tumor growth models with nutrient: from cell density models to free boundary dynamics}
\author{Jian-Guo Liu\thanks{Department of Mathematics and Department of Physics, Duke University (jliu@phy.duke.edu). } \quad Min Tang\thanks{School of Mathematics and Institute of Natural Sciences, MOE-LSC, Shanghai JiaoTong University (tangmin@sjtu.edu.cn) } \quad Li Wang\thanks{Department of Mathematics, Computational and Data-Enabled Science and Engineering Program, State University of New York at Buffalo (lwang46@buffalo.edu). } \quad Zhennan Zhou\thanks{Beijing International Center for Mathematical Research, Peking University (zhennan@bicmr.pku.edu.cn).}}

\begin{document}
\maketitle
\begin{abstract}
In this paper, we study the tumor growth equation along with various models for the nutrient component, including the \emph{in vitro} model and  the \emph{in vivo} model. At the cell density level, the spatial availability of the tumor density $n$ is governed by the Darcy law via the pressure $p(n)=n^{\gamma}$. For finite $\gamma$, we prove some a priori estimates of the tumor growth model, such as boundedness of the nutrient density, and non-negativity and growth estimate of the tumor density. As $\gamma \rightarrow \infty$, the cell density models formally converge to Hele-Shaw flow models, which determine the free boundary dynamics of the tumor tissue in the incompressible limit. We derive several analytical solutions to the Hele-Shaw flow models, which serve as benchmark solutions to the geometric motion of tumor front propagation. Finally, we apply a conservative and positivity preserving numerical scheme to the cell density models, with numerical results verifying the link between cell density models and the free boundary dynamical models.
\end{abstract}

\section{Introduction}
Mathematical modeling and numerical simulations are of growing significance towards understanding cancer development, where the spatial effect has been  one of the most active areas  for modeling the growth of solid tumors. The tumor density can be influenced by a lot of effects, including concentration of nutrients, spatial availability due to contact inhibition, chemical signals, as well as other environmental factors, which yields numerous models for various tumors.     
In order to include spatial effects, two main directions can be found in the literature. One is to use a fluid mechanical view of a tissue, and write down the dynamics of the cell population density \cite{Drasdo2009, Tracy2015,Roose2007}, the other one relies on the fact that the tumor contours are distinguishable, so that one can
use an expanding set $D(t)$ to describe the tumor region \cite{Greenspan1972,Freidman2007,Freidman2008}. Using the asymptotic of a stiff law-of-state pressure, 
the rigorous analysis to build links between these two approaches has been given in \cite{Perthame2014,Tang2013} for those simple cases that the tumor proliferation depends only on contact inhibition. The formal derivation for more complicate cases that take into account other aspects of tumor growth can be found in \cite{Perthamenote}.

According to the setting as in \cite{Perthame2014}, we denote by $n(x,t)$ the cell population density and by $c(x,t)$ the nutrient concentration. The dynamics of the cell population density is governed by the following equation
\begin{equation}\label{eq:n}
\frac{\partial}{\partial t} n - \nabla \cdot \left( n \nabla p(n) \right)=n G(c), \quad x\in \mathbb R^d, \, t\ge0,
\end{equation}
where $p(n)=n^{\gamma}$ ($\gamma$ is a constant) is the pressure and $G(c)$ represents the growth that satisfies the following condition
\begin{equation} \label{eqn:G000}
G'(\cdot)\ge 0, \quad G(0)=0\,.
\end{equation}

The nutrient is governed by the following nutrient equation
\begin{equation}\label{eq:c0}
- \Delta c + \Psi (n,c)=0,
\end{equation}
where $\Psi(n,c)$ is the consumption function which takes different forms in different models.  
As in \cite{PTV2014}, two specific models considered here are the \emph{in vitro} model and the \emph{in vivo} model. For the \emph{in vitro} model, one assumes that the nutrient is constant outside the tumoral region; while the consumption is linear in $c$ inside, thus equation \eqref{eq:c0} reads
\begin{align} 
-\Delta c + \psi(n)c=0, & \quad \mbox{for} \, x\,\in\, D;  \label{eqn:c0-invitro} \\
c=c_B, & \quad   \mbox{for} \, x\,\in\,\mathbb{R} \backslash D, \label{eqn:c1-invitro}
\end{align}
where 
\begin{equation} \label{eqn:D000}
D = \{n(x)>0 \} = \{ p(n)>0 \}\,.
\end{equation}
Here $\psi(n)$ satisfies
\begin{equation} \label{eqn:psi000}
\psi(n)\geq 0  \quad \text {for } n\geq 0, \qquad \text{and} \qquad  \psi(0)=0\,.
\end{equation}
For the \emph{in vivo} model, the nutrient is brought by the vasculature network away from the tumor and diffused to the tissue. In this case, Eqn \eqref{eq:c0} writes
\begin{equation}\label{eq:c}
- \Delta c + \psi(n)c=\chi_{\{ n=0\}} (c_B-c)
\end{equation}
where $\psi$ is the same as in (\ref{eqn:psi000}).

We point out that, in the present paper, $G(c)$ defined in (\ref{eqn:G000}) only takes nonnegative values. Compared with the nutrient models in \cite{Tang2013,Perthame2014,Perthamenote}, we exclude the possibility that $G$ being negative, therefore, no necortic core can appear. Besides, we remark that, since there is no contact inhibition in the growth term, albeit those two nutrient models are of great practical significance, the analysis results in \cite{Perthame2014}, however, applies to neither case directly.

In order to build connections of the cell density model and the free boundary model, the state equation $p(n)$ takes the form $p(n)=n^{\gamma}$ in  \cite{Perthame2014}. The limit when $\gamma \rightarrow \infty$ is considered as the incompressible limit. On the one hand, this limit is physically relevant, it boils down to consider the tumor cell tissue  as an incompressible elastic material in a confined environment. On the other hand, it is mathematically interesting, since the limiting model becomes a Hele-Shaw type free boundary problem.  
To see what happens in the limit of $\gamma \rightarrow \infty$, we multiply the equation \eqref{eq:n} by $\gamma n^{\gamma-1}$ on both sides to get
\begin{align} \label{eqp}
 \frac{\partial}{\partial t} p(n) = |\nabla p(n)|^2 + \gamma  p(n) \Delta p(n)+ \gamma p(n) G(c)\,.
\end{align}
Hence formally we have, when $\gamma \rightarrow \infty$, that $p \rightarrow p_\infty$ with $p_\infty$ solving 
\begin{equation} \label{eqn:p00}
\left\{
\begin{array}{ll}
-\Delta p_\infty = G(c), &\quad  \mbox{in} \quad D_{\infty} (t),  \\
p_\infty=0, & \quad \mbox{on}\quad \partial D_{\infty}  (t).
\end{array}
\right.
\end{equation}
Here 
\begin{equation} \label{eqn:Dinfinity}
D_{\infty} (t)=\{ p_\infty(t)>0\}\,.
\end{equation} 
And $n$ will converge to the weak solution of 
\begin{equation}
\frac{\partial}{\partial t} n_\infty = \nabla \cdot \left( n_\infty \nabla p_\infty  \right)+n_\infty G(c)\,,
\end{equation}
wherein the limit density $n_{\infty}$ satisfies $0\leq n_\infty\leq 1$ and $n_\infty=1$ in $D_{\infty}(t)$. Note here the difference between $D$ in (\ref{eqn:D000}) and $D_{\infty}$ in (\ref{eqn:Dinfinity}): the former one is for finite $\gamma$ when $n$ and $p(n)$ have the same support, whereas the latter is when taking $\gamma$ to infinity and $n_\infty$ may have a larger support than $p_\infty$. For a general class of initial conditions, see \cite{Perthamenote}, $n_\infty$ converges to a patch function $\chi_{D_{\infty}(t)}$ as time goes on, and the velocity of the free boundary $\partial D_\infty$ is $v=-\nabla p_\infty$. In this case, the supports of $n_\infty$ and $p_\infty$ coincide. 

Then, in the \emph{in vitro} model, equation \eqref{eq:c0} becomes
\begin{align}
-\Delta c + \psi(1)c=0, & \quad \mbox{for} \, x\,\in\,D_\infty(t);  \\
c=c_B, & \quad   \mbox{for} \, x\,\in\, \{p_\infty(t)=0\}.
\end{align}
And for the \emph{in vivo} model, equation \eqref{eq:c0} becomes
\begin{equation}\label{eq:c}
- \Delta c + \psi(1) \chi_{D(t)}c=\chi_{\{ p(t)=0\}} (c_B-c).
\end{equation}

In this paper,  
three different nutrient dependence are considered: 1) $G(c)$ is a constant in the whole domain; 2) \emph{in vitro}; 3) in \emph{vivo}.
The main contribution of this paper is two-fold. One is to provide some a priori estimates of the non-negativity and global boundedness of the nonlinear parabolic-elliptic system. It is important to note that, different from the models in \cite{Perthame2014}, the cell growth is not prohibited by the contact inhibition, which is the case for tumor cells in vitro, thus there exhibits no maximum pressure, and in the nutrients models, $\Psi$ is no longer necessarily smooth functions of $n$ or $p$. Therefore, the proof of the non-negativity and global boundedness is not as straightforward as in \cite{Perthame2014}. The other is to derive some benchmark analytical solutions for multi-dimensional front dynamics, a geometric motion of the limiting free boundary model. These solutions compare favorably with the numerical solutions to cell density model, which to some extend, verity this singular limit. To solve the cell density model numerically, we adopt a recently proposed numerical scheme for sub-critical Keller Segel equations \cite{LiuWangZhou} to the tumor growth models, which is conservative in the spatial flux due to pressure, positivity preserving, and free from a nonlinear solver.

There exist other models in the literature that can connect time dynamic density model with the free boundary model, for example the threshold dynamics method introduced in \cite{Osher1994} found similar connections and is used to simulate the motion by mean curvature flow \cite{Ruuth1998,Tsai2008,Xu2016}. 
The incompressible limit of the tumor growth model is interesting not only because it provides the link between different model types, but also it provides a possible tool to simulate and approximate the free boundary problems. {In the numerical part, 2D geometric motions of the free boundary models are investigated as the limits of cell density models as $\gamma \rightarrow \infty$}. 

The organization of this paper is as follows. We prove the non-negativity and global boundedness of the cell population density model with finite $\gamma$ in section 2. Some multidimensional geometric front dynamics are derived analytically in section 3. In section 4 and 5, we introduce the adopted numerical scheme for the tumor growth models, verify the analytical results found in section 3 by simulating the cell density models, and present some worthy geometric motions of the limiting free boundary model.


\section{Properties of the PDE models}
In this section, we commence a study of various basic estimates of the solution to the tumor-nutrient models with fixed $\gamma$. The properties we will cover include the non-negativity and global boundedness of the tumor density $n$ and nutrient density $c$, and limited growth for the total mass of the tumor. 

In \cite{Perthame2014}, a general class of parabolic-parabolic systems of the tumor cell density and the nutrient concentration were studied, where the coupling functions $G$ and $\Psi$ are assumed to be smooth functions of $p$. However,  in the \emph{in vitro} model and the {in vivo} model, the nutrient density functions are governed by elliptical equations with moving boundary conditions or nonsmooth dependence on the cell density function $n$, and hence the analysis in  \cite{Perthame2014} cannot be directly extended to the models we study. 

First,  the non-negativity for $n(x,t)$ is given by the following theorem. 
\begin{theorem}\label{thm:n}
For the tumor growth model \eqref{eq:n}, if initially $n(x,0)=n_0(x)\ge0$, and $G(c)\in[0, G_m]$ for some $G_m>0$, then $n(x,t) \ge 0$ for all $t\ge0$ and $x \in \mathbb R^d$, regardless of the specific form of $G(c)$.   
\end{theorem}

This result is an immediate consequence of the comparison principle  (see e.g. Proposition 4.5 in \cite{Perthamenote}) since $n(x,t)\equiv 0 $ is a trivial solution to \eqref{eq:n}, and hence we omit the proof in this work.  We emphasis that, this property is independent of the growing factor $G(c)$, and thus naturally applies to all three tumor-nutrient models we listed in the previous section.

Next, we consider the change in the total mass of $n$. With sufficient nutrient $G(c) \equiv G_0$, we immediately get the exponential growth in time. Indeed, integrate \eqref{eq:n} against $x$ over $\mathbb R^n$, we have
\[
\frac{d}{dt}\int_{\mathbb R^d}  n (x,t) dx =  G_0 \int_{\mathbb R^d}  n(x,t) dx \,, 
\]
which readily implies
\[
\|n(\cdot, t) \|_{L^1}   =  \exp(t G_{0})\|n(\cdot, 0) \|_{L^1} \,,
\]
thanks to the non-negativity of $n$ in Theorem~\ref{thm:n}. For the other two nutrient models, we need to prove the global boundedness of the nutrient functions first.  Note that the nutrient densities are self-consistently determined by the cell density with moving support, we assume that this support propagates with finite speed. It is well known that for porous media equations, the support of the density function expands with bounded speed, see, for example, \cite{AronsonCaffarelliKamin}, and similar estimates have been derived for some tumor growth models in \cite{Perthame2014}.  In the paper, due to the strong similarities to the tumor growth models in \cite{Perthame2014}, we choose to skip this proof, and focus on the estimates on the nutrient models. 
 
Having established the non-negativity of the cell density  model, we next demonstrate the boundedness on $c(x,t)$. Recall the \emph{in vitro} model in the following, 
\begin{align}
\frac{\partial}{\partial t} n =\nabla \cdot (n \nabla p(n)) +n G(c), & \quad x \in \mathbb R^d, \label{eq:ln} \\
-\Delta c + \psi(n)c=0, & \quad x \in D(t), \label{eq:lc} \\
c(x,0)=c_B>0, &  \quad x \in \mathbb R^d \backslash D(0), \label{BCc}
\end{align}
with the initial condition 
\begin{equation}
n(x,0)=n_0(x) \ge 0, \label{ICn}
\end{equation}
where $n_0(x)$ is a compactly supported function, and $D(t)$ is defined in (\ref{eqn:D000}). We now show the following lemma.  
\begin{lemma}\label{lem:c}
In the {in vitro} model \eqref{eq:ln}--\eqref{ICn}, $0 \le c(x,t) \le c_B$ for $t\ge0$ and $ x \in\mathbb R^d$.
\end{lemma}
\begin{proof}
Write $c(x,t)= c_+(x,t) - c_-(x,t)$, where $c_\pm(x,t) \ge 0$ denote the positive part and the negative part of $c(x)$, respectively.  Notice that at the boundary $c(x,t)|_{\partial D(t)}=c_B>0$, thus $c_-(x,t)|_{\partial D(t)}=0$. Then multiply equation \eqref{eq:lc} by $-c_- (x,t)$, and integrate over $D(t)$, we have, upon integration by parts, 
\[
\int_{D(t)} |\nabla c_-|^2 + \psi(n) c^2_-  d x=0.
\]
Notice that by Theorem \ref{thm:n}, the cell density stays nonnegative and thus $\psi(n) \ge 0$. 
Then the above equation implies $c_-(x,t)=0$ in $D(t)$, and thus $c(x,t)\ge 0$ in $D(t)$. Hence,  $c(x,t)\ge 0$   in $\mathbb R^d$. 

Now from (\ref{eq:lc}), we have in $D(t)$, 
\[
\Delta c =\psi(n) c \ge 0.
\] Then, by maximum principle, we have
\[
c(x,t)|_{D(t)} \le c(x)|_{\partial D(t)} =c_B,
\]
which indicates that $c(x,t) \le c_B$ in $\mathbb R^d$. 
\end{proof}


Next we turn our attention to the \emph{in vivo} model
\begin{align}
\frac{\partial}{\partial t} n =\nabla \cdot (n \nabla p(n)) +n G(c), & \quad x \in \mathbb R^d, \label{eq:ln2} \\
-\Delta c + \psi(n)c=\chi_{\{ \mathbb R^n \backslash D(t) \}} (c_B -c), & \quad x \in \mathbb R^d, \label{eq:lc2} \\
n(x,0)=n_0(x)\ge 0,\qquad & c(\pm\infty,t)= c_B>0 , \label{ICn2}
\end{align}
Again, we assume that $n_0 $ is compactly supported. We show the boundedness of $c(x,t)$ in the following.
\begin{lemma}\label{lem:c2}
In the {in vivo} model \eqref{eq:ln2}--\eqref{ICn2},  for $t\ge0$ and $ x \in\mathbb R^d$, $0 \le c(x,t) \le c_B$.
\end{lemma}
\begin{proof}
Similar to the proof in Lemma \ref{lem:c}, by Theorem \ref{thm:n}, the cell density stays nonnegative and $\psi(n) \ge 0$. 

We write $c(x,t)= c_+(x,t) - c_-(x,t)$, where $c_\pm(x,t) \ge 0$ denote the positive part and the negative part of $c(x)$, respectively. Due to assumption of $c(x,t)$ at infinity,  $c_-(\pm\infty,t)=0$. Multiply equation \eqref{eq:lc2} by $-c_- (x,t)$, and integrate over $\mathbb R^n$, we have, upon integration by parts,
\[
\int_{\mathbb R^n} \left( |\nabla c_-|^2+ \psi(n) c_-^2 \right) dx  + \int_{\mathbb R^n \backslash D(t)} \left( c_B c_- + c_-^2 \right) dx =0\,
\]
which implies $c_-(x,t)=0$ in $\mathbb R^n$, and thus $c(x,t)\ge 0$ in $\mathbb R^n$. Next, we show that there exits an upper bound for $c(x,t)$. In $\mathbb R^n \backslash D(t)$, due to the boundary conditions that $c(\pm\infty,t)=c_B$, if we further assume:  
\[
\max_{x \in \mathbb R^n \backslash D(t)} c(x,t) =c_M > c_B,
\]
then there exists $x_0 \in \mathbb R^n \backslash  D(t)$, such that
\[
\Delta c(x_0,t) \le 0, \quad \mbox{and} \quad c(x_0,t)=c_M.
\]
This implies,
\[
-\Delta c(x_0,t) \ge 0 > c_B-c_M.
\]
So equation \eqref{eq:lc2} is violated at this point. Therefore, when $x \in \mathbb R^n \backslash D(t)$, $c(x,t) \le c_B$. In $D(t)$, 
\[
\Delta c =\psi(n) c \ge 0.
\] By maximum principle, we have
\[
c(x,t)|_{D(t)} \le c(x,t)|_{\partial D(t)}.
\]
This clearly shows that, by continuity of $c(x,t)$ crossing $\partial D(t)$, $c(x,t) \le c_B$ in $\mathbb R^d$, which completes the proof.
\end{proof}

As an immediate result, in both models, we have the following estimate in the growth of the total mass, 
\begin{equation}
\|n(\cdot, t) \|_{L^1}  \le  \exp(t G_{m}) \|n(\cdot, 0) \|_{L^1},
\end{equation}
where $G_m=G(c_B)$.


\section{Explicit solutions of the Hele-Shaw models} \label{sec:analytical solutions}

Assume that $n$ starts with a characteristic function, then it is expected that it remains so when $\gamma$ goes to infinity and thus the cell density model converges to the Hele-Shaw flow \cite{Perthame2014, Perthamenote}. In this and the next sections, we would like to build a more concrete connection between these two models. Particularly, we explicitly work out the analytical solutions of the Hele-Shaw type equations for the three tumor nutrient models in this section, which will be compared with  numerical solutions to the cell density models obtained in Section \ref{sec:num}. The analytical solutions we obtain in this section will also  serve as a benchmark for our future research. 

\subsection{Radial symmetric solution with constant nutrient in multi-dimensions}
Consider the tumor growth model with infinitely sufficient nutrient
\begin{equation} \label{eqn:nG0}
\frac{\partial}{\partial t} n = \nabla \cdot \left( n \nabla p(n) \right)+n G_0.
\end{equation}
We recall here for convenience that $p(n)=n^{\gamma}$. As explained in the introduction, in the limit of $\gamma \rightarrow \infty$, we have $p \rightarrow p_\infty$, and the model formally becomes the Hele-Shaw geometric model. Specifically, it takes the form
\begin{equation} \left\{
\begin{array}{ll}
-\Delta p_\infty = G_0, &\quad  \mbox{in} \quad D_\infty (t),   \label{eq:HeleP0}\\
p_\infty=0, & \quad \mbox{on}\quad \partial D_\infty (t),
\end{array}
\right.
\end{equation}
where $D_\infty(t)=\{ p_\infty(t)>0\}$. The boundary of $D_\infty(t)$ moves speed $v=-\nabla p_\infty \cdot \hat n$ along the normal direction, where $\hat n (x,t)$ is the outer unit normal vector at the boundary. And $n_\infty$ is a weak solution to 
\begin{equation}\label{n_inf0}
\frac{\partial}{\partial t} n_\infty = \nabla \cdot \left( n_\infty \nabla p_\infty  \right)+n_\infty G_0.
\end{equation}

In what follows, we confine ourselves to the radial symmetric case and derive the analytical solutions explicitly for several specific examples. Let $r$ be the radial variable; then (\ref{eq:HeleP0}) rewrites
\begin{equation} \label{eqn:TN-r}
\partial_t n = \frac{1}{r^{d-1}} \frac{\partial}{\partial r} \left( n r^{d-1}  \frac{\partial }{ \partial r} p \right) + nG_0 \,,
\end{equation}
and (\ref{eq:HeleP0}) becomes 
\begin{equation} \label{eqn:Hele-r}
-\frac{1}{r^{d-1}} \frac{d}{d r} \left( r^{d-1} \frac{d}{d r} p_\infty \right)=G_0 \,,
\end{equation}
and the expansion speed takes the form
\[
v = -\frac{\partial p_\infty }{\partial r} \hat{r} \cdot \hat{n}\,.
\]
Here both $n$ and $p$ now depend on $r$, and $d$ denotes the dimension.

\vspace{0.5cm}
\hspace{-0.65cm}{\bf Example 1: An expanding ball.\,}
We choose the initial condition to be the characteristic function of a ball with radius $R_0$ centered at origin, i.e.,
\[
n(x,0)=\chi_{B_{R_0}}.
\]
then it is expected that as $\gamma$ goes to infinity,  $n(x,t)$ converges to $n_\infty (x,t)=\chi_{B_{R(t)}}$. Now it amounts to determine how $R(t)$ changes with time. We will explore this dynamics in the viewpoint of both the tumor growth model and the limiting Hele-Shaw flow model, and show that they both lead to the same expansion speed for the tumor. 

Firstly, in the tumor growth model (\ref{n_inf0}), we integrate both sides over $\mathbb R^d$, and denote 
$m(t)= \int_{\mathbb R^d} n_\infty (x,t) d x$, then we get $m(t)= m(0) e^{G_0 t}$.
With the radial symmetric assumption, this solution implies
\begin{equation} \label{eqn:R-ball}
R(t)= R(0) e^ {  {G_0 t}/{d}}\,,
\end{equation}
which leads to the expansion speed
\begin{equation} \label{eqn:V-ball}
\partial_t R = \frac{G_0}{d} R.
\end{equation}

On the other hand, for the Hele-Shaw flow model (\ref{eqn:Hele-r}), we see that
\[
 \frac{d}{d r} \left( r^{d-1} \frac{d}{d r}P\right) = - r^{d-1}G_0.
\]
Integrate it with respect to $r$ from $0$ to $R$, one gets
\[
R^{d-1}\frac{d}{d R}P= - \frac{1}{d} R^d G_0,
\]
which implies 
\begin{equation} \label{eq:PR}
\frac{d}{d R}P = - \frac{G_0}{d} R .
\end{equation}
Therefore, the expansion speed is
\begin{equation*}
v=-\nabla p_\infty \cdot \hat n= - \frac{d P}{d R}= \frac{G_0}{d} R\,,
\end{equation*}
which agree with the speed (\ref{eqn:V-ball}) derived from the dynamical tumor growth model. 

Also, we conclude from \eqref{eq:PR} that 
\[
P= -\frac{G_0}{2d} R^2+a,
\]
and the integration constant $a$ can be determined by the fact that $P(R(t))=0$, and thus
\begin{equation*}
P(r)=-\frac{G_0}{2d} r^2 + \frac{G_0}{2d} R(t)^2, \quad r\le R(t).
\end{equation*}

\vspace{0.5cm}
\hspace{-0.7cm} {\bf Example 2: a single-annulus in dimension 2.\,} As the second example, we consider the case when $D_\infty(t)$ has an annulus shape with inner radius $r_-$ and outer radius $r_+$. In this case, we can not derive the speed for the two boundaries from the tumor growth model but only from the limit Hele-Shaw flow model. Recall (\ref{eqn:Hele-r}), then the solution $p_\infty(r)$ takes the form
\begin{equation} \label{eqn:p-sa}
p_\infty(r)=\left\{ 
\begin{array}{ll}
-\frac{G_0}{4} r^2+a \ln r + b, & d=2 ;\\
-\frac{G_0}{2d} r^2+ \frac{a}{2-d} {r^{2-d}}+b, & d \ge 3 .
\end{array}
\right.
\end{equation}
Here both $a$ and $b$ will be determined by the fact that $p_\infty(r_-)=p_\infty(r_+)=0$. In particular, 
when $d=2$, we have
\begin{equation*}
a=G_0\frac{ r_+^2 - r_-^2}{4 \left( \ln r_+ - \ln r_- \right)}\,, \qquad b= -G_0\frac{r_+^2 \ln r_- - r_-^2 \ln r_+}{4\left( \ln r_+ - \ln r_- \right)}.
\end{equation*}
The case with $d>3$ can be derived in exactly the same manner and we omit its detailed form in this paper. To lighten the notation, we let $m$ denote the total mass
\[
m = B_d(r_+)- B_d(r_-)\,,
\]
where $B_d(r) = \frac{\pi ^{d/2}}{\Gamma \left( \frac{d}{2} +1\right) } r^d$ is the volume of a ball in $\mathbb R^d$ with radius $r$. Then one sees from equation (\ref{eqn:TN-r}), upon integrating in $\mathbb R^d$, that
\begin{equation} \label{eqn:mt-sa}
m(t) = m(0) e^{G_0 t} \,.
\end{equation}
In $d=2$, $m$ simply reduces to $m = \pi (r_+^2 - r_-^2)$. 

Given the form of $p_\infty$ in (\ref{eqn:p-sa}), one immediately gets the moving speed. Specifically, at the inner boundary, we have 
\begin{equation*}
v\big|_{r=r_-} = -p'_\infty \hat{n} \big|_{r_-} = p'_\infty (r_-)=  -\frac{G_0} {2} r_-  + \frac{G_0 m}{4 \pi r_- \left( \ln r_+ - \ln r_- \right)}\,;
\end{equation*}
whereas in the outer boundary, we have 
\begin{equation*}
v\big|_{r=r_+} =-p'_\infty \hat{n} \big|_{r_+}  =  - p'_\infty(r_+)=  \frac{G_0}{ 2} r_+ -\frac{G_0 m}{4 r_+ \pi \left( \ln r_+ - \ln r_- \right)}.
\end{equation*}
Note carefully here that the inner boundary moves at speed $v\big|_{r_-}$ in the negative direction along the radius, and the outer boundary moves at the speed of $v\big|_{r-+}$ in the positive direction along the radius. Therefore, we have the following results concerning the change in radius $r_+$ and $r_-$:
\begin{align}
\partial_t r_- & =- v|_{r_-} =  \frac{G_0}{ 2} r_- - \frac{G_0 m}{4 \pi r_- \left( \ln r_+ - \ln r_- \right)},
\label{eqn:r-sa} \\
\partial_t  r_+ & = v|_{r_+}=  \frac{G_0}{ 2} r_+ - \frac{G_0 m}{4 r_+ \pi \left( \ln r_+ - \ln r_- \right)}.
\label{eqn:r+sa}
\end{align}
Moreover, one can easily check that 
\[
\partial_t  m = 2 \pi \left(r_+ \dot r_+ - r_- \dot r_- \right) =2 \pi \frac{G_0}{2} (r^2_+-r^2_-)=G_0 m \,,
\]
which recovers the exponential growth of the total mass as displayed in (\ref{eqn:mt-sa})

\vspace{0.5cm}
\hspace{-0.7cm} {\bf Example 3: a double-annulus in dimension 2.\,} In this example, we extend the single annulus into a double annulus shape with four boundaries $r_1$, $r_2$, $r_3$ and $r_4$, where $r_1< r_2$ characterize the inner annulus and $r_3 < r_4$ defines the outer annulus. Then similar to the previous example, we can only compute the front propagation speed via the limit model (\ref{eqn:Hele-r}). Indeed, from (\ref{eqn:p-sa}), one has 
\[
p_\infty(r) = -\frac{r^2}{4} G_0 + a \ln r + b\,,
\]
where $a$ and $b$ are determined by the boundary conditions. Specifically, for the inner annulus, the boundary conditions are 
\[
p(r_1) = p(r_2) = 0\,,
\]
which leads to 
\begin{equation*}
p_\infty(r) =  -\frac{r^2}{4} G_0 + G_0 \frac{r_2^2 - r_1^2}{4 (\ln r_2 - \ln r_1)} \ln r  - G_0 \frac{r_2^2 \ln r_1 - r_1^2 \ln r_2}{4(\ln r_2 - \ln r_1)}   \qquad r_1 \leq r \leq r_2\,.
\end{equation*}
Therefore, $r_1$ and $r_2$ change according to the following two equations
\begin{align*}
\partial_t r_1 & =- v|_{r_1} = -p_\infty'(r_2)  =  \frac{G_0}{ 2} r_1 - \frac{G_0 (r_2^2 - r_1^2)}{4 r_1 \left( \ln r_2 - \ln r_1 \right)}, \\
\partial_t  r_2 & = v|_{r_2}=  -p_\infty'(r_2) =  \frac{G_0}{ 2} r_2 - \frac{G_0(r_2^2 - r_1^2)}{4 r_2  \left( \ln r_2 - \ln r_1\right)}.
\end{align*}
Likewise, $r_3$ and $r_4$ satisfy the following equation
\begin{align*}
\partial_t r_3 & =- v|_{r_3} = -p_\infty'(r_3)  =  \frac{G_0}{ 2} r_3 - \frac{G_0 (r_4^2 - r_3^2)}{4 r_3 \left( \ln r_4 - \ln r_3 \right)}, \\
\partial_t  r_4 & = v|_{r_4}=  -p_\infty'(r_4) =  \frac{G_0}{ 2} r_4 - \frac{G_0(r_4^2 - r_3^2)}{4 r_4  \left( \ln r_2 - \ln r_3\right)}.
\end{align*}
And $p_\infty(r)$ for $r \in [r_3, r_4]$ takes the form 
\begin{equation*}
p_\infty(r) =  -\frac{r^2}{4} G_0 + G_0 \frac{r_4^2 - r_3^2}{4 (\ln r_4 - \ln r_3)} \ln r  - G_0 \frac{r_4^2 \ln r_3 - r_3^2 \ln r_4}{4(\ln r_4 - \ln r_3)}   \qquad r_3 \leq r \leq r_4\,.
\end{equation*}


\subsection{1D radial symmetric model with linear growth function }
In this section, we assume that the growing factor $G(c)$ is a linear function in $c$
\begin{equation}\label{eqn:G-linear}
G(c) = G_0 c\,, \quad G_0>0
\end{equation}
so that it satisfies the conditions (\ref{eqn:G000}). Then the tumor growth model (\ref{eq:n}) in 1D reduces to 
\begin{equation*}
\partial_t n = \partial_x  \left( n \partial_x p(n) \right)+ G_0 c n\,, \quad p(n) = n^\gamma\,.
\end{equation*}
In the limit of $\gamma \rightarrow \infty$, we have the limit density $n_\infty$ solving
\begin{equation*}
\frac{\partial}{\partial t} n_\infty = \partial_x  \left( n_\infty \partial_x p_\infty  \right)+n_\infty G_0 c.
\end{equation*}
and $p_\infty$ in (\ref{eqn:p00}) satisfying 
\begin{equation}  \label{eq:HeleP}
\left\{
\begin{array}{ll}
-\partial_{xx} p_\infty = G_0c, &\quad  \mbox{in} \quad D_\infty (t),   \\
p_\infty=0, & \quad \mbox{on}\quad \partial D_\infty (t),
\end{array}
\right.
\end{equation}
where $D_\infty(t)=\{ p_\infty(t)>0\}$. The free boundary of $D$ moves with normal velocity 
\begin{equation}\label{eqn:speed}
v=- \partial_x p_\infty \cdot \hat n
\end{equation}
with $\hat n (x,t)$ being the unit outer normal vector to the boundary. In the following two examples, we derive the analytical solutions for the limiting models obtained from two different cases: {\it in vitro} and {\it in vivo}.

\vspace{0.5cm}
\hspace{-0.7cm}{\bf Example 4: 1D {\it in vitro} model.\,} In the 1D \emph{in vitro} models, equations \eqref{eqn:c0-invitro} \eqref{eqn:c1-invitro} become
\begin{align*}
-\partial_{xx} c + \psi(1)c=0, & \quad \mbox{for} \, x\,\in\,D_\infty(t);  \\
c=c_B, & \quad   \mbox{for} \, x\,\in\, \mathbb R \backslash D_\infty(t) \,,
\end{align*}
and  we have formally assumed 
 that on $D_\infty(t)$, $n \equiv 1$ if initially $n$ is a characteristic function \cite{Perthame2014, Perthamenote}.  

Now assume $\psi(n)=n$ for simplicity, then at a certain time $t$ (we hereafter suppress the $t$ dependence whenever it does not cause any confusion), we have 
\[
 - \partial_{xx} c + c =0, \qquad x \in [-R(t),R(t)]\,.
\]
Since $c$ is symmetric with respect to the origin, we have $\partial_x c(0) =0$, which implies that
\[
c(x)= a \cosh(x).
\] 
Here $a$ is obtained from the boundary condition $c(\pm R(t))=c_B$:
\[
a= \frac{c_B}{\cosh(R(t))}.
\]
Putting together, we have 
\[
c= \left\{
\begin{array}{ll}
\frac{c_B }{\cosh(R(t))} \cosh(x), & x \in [ -R(t),R(t)]; \\
c_B, & x \notin   [ -R(t),R(t)] .
\end{array}
\right.
\]
To proceed, plugging the above solution for $c$ into the $p_\infty$ equation (\ref{eq:HeleP}), we get
\[ 
-\partial_{xx} p_\infty = G_0c = \frac{c_B G_0}{\cosh(R(t))} \cosh(x), \qquad x \in [ -R(t),R(t)]\,,
\]
whose general solution is given by
\[
p_\infty = - \frac{c_B G_0}{\cosh(R(t))} \cosh(x) + ax+b.
\]
Again, by symmetry, one has $\partial_x p(0)=0$, which leads to $a=0$. Then the boundary condition $p_{\infty}(\pm R(t))=0$ gives rise to $b=c_B G_0$. Therefore, we have
\begin{equation} \label{eqn:p-1D-invitro}
p_\infty= \left\{
\begin{array}{ll}
-\frac{c_B G_0 }{\cosh(R(t))} \cosh(x) + c_B G_0, & x \in [ -R(t),R(t)]; \\
0, & x \notin   [ -R(t),R(t)] .
\end{array}
\right.
\end{equation}
Then the propagation speed of the $R(t)$ can be obtained using (\ref{eqn:speed})
\[
v(R(t)) = -p'_\infty (R(t))= c_B G_0 \tanh (R(t)).
\]
and thus 
\begin{equation} \label{eqn:R-invitro}
\partial_t R(t) =c_B G_0 \tanh (R(t))\,.
\end{equation}
As $R(t) \rightarrow \infty$, one sees that the limiting speed is $c_B G_0$.

\vspace{0.5cm}
\hspace{-0.7cm}{\bf Example 5: 1D {\it in vivo} model.\, } We now repeat the calculation for the \emph{in vivo} model, in which nutrient varies according to 
\begin{equation*}\label{eq:c}
- \partial_{xx} c + \psi(1) \chi_{D(t)}c=\chi_{\{ p(t)=0\}} (c_B-c).
\end{equation*}
With the same assumptions as in the previous example, we have at a certain time $t$, 
\[
 - \partial_{xx} c + c =0, \qquad x \in [-R(t), R(t)]\,.
\]
Along with $\partial_x c(0) =0$ that comes from the symmetric assumption, we get
\[
c(x)= a_0 \cosh(x).
\] 
Now comes the difference from the previous example: we cannot specify the constant $a_0$ with the boundary condition. Instead, we have
\[
-\partial_{xx}c = c_B -c, \qquad  x>R(t)\,,
\]
whose general solution is given by 
\[
c=c_B + a_1 e^{-x}+ a_2 e^{x}.
\]
With the far field assumption $c \rightarrow c_B$ as $x \rightarrow \pm \infty$, we obtain $a_2=0$. Then by the continuity of both $c$ and $\partial_x c$ at $x=R(t)$, we get
\[
a_0 = \frac{c_B}{e^{R(t)}},\quad a_1=-c_B \sinh(R(t)).
\]
In summary, 
\[
c= \left\{
\begin{array}{ll}
\frac{c_B }{e^{R(t)}} \cosh(x), & x \in [ -R(t),R(t)]; \\
c_B -c_B \sinh(R(t)) e^{-|x|}, & x \notin   [ -R(t),R(t)] .
\end{array}
\right.
\]
As before, plugging the expression of $c$ into (\ref{eq:HeleP}) to get $p_\infty$
\[ 
-\partial_{xx} p_\infty = G_0c = \frac{c_B G_0}{e^{R(t)}} \cosh(x), \quad x\in [-R(t),R(t)]
\]
whose general solution is given by
\[
p_\infty = - \frac{c_B G_0}{e^{R(t)}} \cosh(x) + ax+b.
\]
Then symmetry implies $\partial_x p(0)=0$, which further leads to $a=0$. And the boundary condition $p_{\infty}(\pm R(t))=0$ implies $b=c_B G_0 \cosh(R(t)) e^{-R(t)}$. Altogether, we get
\begin{equation}  \label{eqn:p-invivo}
p_\infty= \left\{
\begin{array}{ll}
-\frac{c_B G_0 }{e^{R(t)}} \cosh(x) + \frac{c_B G_0 }{e^{R(t)}} \cosh(R(t)) , & x \in [ -R(t),R(t)]; \\
0, & x \notin   [ -R(t),R(t)] .
\end{array}
\right.
\end{equation}
And the propagation speed of $R(t)$ is obtained by direct calculation
\begin{equation*} 
v(R(t)) = -p'_\infty (R(t))=c_B G_0 \frac{\sinh(R(t))}{e^{R(t)}}= c_B G_0 \frac{\cosh(R(t))}{e^{R(t)}} \tanh (R(t)) \le c_B G_0 \tanh (R(t)) \,,
\end{equation*}
and thus 
\begin{equation} \label{eqn:R-invivo-0}
\partial_t R(t) =c_B G_0 \frac{\sinh(R(t))}{e^{R(t)}},
\end{equation}
In view of the above result, we notice that the propagation speed in the \emph{in vivo} model is slower than that in the \emph{in vitro} model. Moreover, as $R(t) \rightarrow \infty$, the limiting speed is $\frac{1}{2} c_B G_0$, which is a half of the limiting speed in the {\it in vitro} model.

\subsection{2D radial symmetric model with linear growth}
As in the last section, we consider linear growth function (\ref{eqn:G-linear}) but in 2D radial symmetric case. Then (\ref{eq:n}) simplifies to 
\begin{equation*}
\partial_t n =  \frac{1}{r} \partial_r \left( nr \partial_r p(n)\right) + G_0 c n, \qquad p(n) = n^\gamma\,
\end{equation*}
and its limit reads
\begin{equation*}
\partial_t n_\infty = \frac{1}{r}\partial_r \left( r\partial_r p_\infty \right) + n_\infty G(c)\,,
\end{equation*}
where $p_\infty$ satisfies 
\begin{equation} \label{eqn:limit25}
\left\{ \begin{array}{cc} -\frac{1}{r}\left( r \partial_r  p_\infty \right) = G(c)  & \quad   \text{ in } \quad D(t) \,,
\\ p_\infty = 0 & \quad \text{ on } \quad \partial D(t)\,. \end{array} \right.
\end{equation}
The equation for $c$ varies depending on the model we considered. In the following two examples, we provide analytical solution for the limiting system.

\vspace{0.5cm}
\hspace{-0.65cm}{\bf Example 6: 2D radial symmetric {\it in vitro} model.\, } 
In the {\it in vitro} model, we have 
\begin{align}
-\frac{1}{r} \partial_r(r \partial_r c) + \psi(1)c=0, & \quad \mbox{for} \, x\,\in\,D(t);   \\
c=c_B, & \quad   \mbox{for} \, x\,\in\, \mathbb R^2 \backslash D(t).
\end{align}
For simplicity, we use $\psi(n) = n$ from now on. We also assume that the initial density $n$ is a characteristic function with radial symmetry, i.e., $n_\infty(x,0) = \chi_{B_{R_0}}$, and we expect the density remains a characteristic function with a moving boundary $n_\infty = \chi_{B_{R(t)}}$.

For fixed $t$, (we thus suppress the $t$ dependence in the calculation in the following) when $x \in B_{R(t)}$, we have 
\[
 - \frac{1}{r}  \partial_{r} ( r \partial_r c ) + c =0.
\]
The boundedness of $c$ at $r=0$ implies the following general solutoin
\[
c(r)= a I_0(r), 
\] 
where $I_m (r)$ is the modified Bessel function of the first kind.
The undetermined coefficient $c$ comes from the boundary condition at $c( R(t))=c_B$, which leads to
\[
a= \frac{c_B}{I_0 (R(t))}.
\]
Therefore, 
\[
c= \left\{
\begin{array}{ll}
\frac{c_B }{I_0 (R(t))} I_0 (r), & r \in [0,R(t)]; \\
c_B, & r > R(t) .
\end{array}
\right.
\]
To proceed, plugging the solution $c$ into (\ref{eqn:limit25}), then we have, for $x \in B_{R(t)}$
\[ 
-\frac{1}{r}  \partial_{r} ( r \partial_r p_\infty )  =\frac{c_B G_0}{I_0 (R(t))} I_0 (r),
\]
whose general solution is given by
\[
p_\infty = - \frac{c_B G_0}{I_0 (R(t))} I_0 (r)+ a \ln r+b.
\]
The boundedness of $p_\infty$ at $r=0$ implies $a=0$, and the boundary condition $p_{\infty}(\pm R(t))=0$ implies $b=c_B G_0$. In sum, we get
\[
p_\infty= \left\{
\begin{array}{ll}
 - \frac{c_B G_0}{I_0 (R(t))} I_0 (r)+ c_B G_0, & r \in [0,R(t)]; \\
0, & r > R(t).
\end{array}
\right.
\]
Then the propagation speed of the $R(t)$ is
\begin{equation*}
v(R(t)) = -p'_\infty (R(t))= c_B G_0  \frac{I_1 (R(t))}{I_0 (R(t))}\,,
\end{equation*}
and thus 
\begin{equation}\label{eqn:Rtt-vitro}
\partial_t R(t) = c_B G_0  \frac{I_1 (R(t))}{I_0 (R(t))}\,.
\end{equation}
Note that limiting speed is $c_B G_0$ as $R(t) \rightarrow \infty$.

\vspace{0.5cm}
\hspace{-0.65cm}{\bf Example 7: 2D radial symmetric {\it in vivo} model.\, } 
We now repeat the calculation for the \emph{in vivo} model: 
\begin{equation*}\label{eq:c}
 - \frac{1}{r}  \partial_{r} ( r \partial_r c ) + \psi(1) \chi_{D(t)}c=\chi_{\{ p(t)=0\}} (c_B-c).
\end{equation*}
With the same assumptions as in the previous section,  
for fixed $t$, and when $x \in B_{R(t)}$, we have 
\[
  - \frac{1}{r}  \partial_{r} ( r \partial_r c )  + c =0.
\]
The boundedness of $c$ at $r=0$ implies the following solution,
\[
c(x)= a_0 I_0 (r).
\] 
However, unlike the previous case, we can not specify the constant $a_0$ with the right boundary condition. Instead, we have, for $x>R(t)$,
\[
 - \frac{1}{r}  \partial_{r} ( r \partial_r c ) = c_B -c,
\]
and thus the general solution is given by
\[
c=c_B + a_1 K_0(r)+ a_2 I_0(r),
\]
where $K_m(r)$ denotes the modified Bessel function of the second kind. 
With the far field assumption $c \rightarrow c_B$ as $x \rightarrow \pm \infty$, we know $a_2=0$. By continuity of $c$ and $\partial_x c$ at $x=R(t)$, we get
\[
a_0 = \frac{c_B K_1(R)}{K_0(R)I_1(R)+K_1(R)I_0(R)} ,\quad a_1=-\frac{c_B I_1(R)}{K_0(R)I_1(R)+K_1(R)I_0(R)} .
\]
Therefore,
\[
c= \left\{
\begin{array}{ll}
\frac{c_B K_1(R)}{K_0(R)I_1(R)+K_1(R)I_0(R)}  I_0(r), & r \in [0,R]; \\
c_B -\frac{c_B I_1(R)}{K_0(R)I_1(R)+K_1(R)I_0(R)} K_0(r), & r>R .
\end{array}
\right.
\]
Plugging it to (\ref{eqn:limit25}), then for $x \in B_{R(t)}$, we have 
\[ 
-\frac{1}{r}  \partial_{r} ( r \partial_r p_\infty )  = G_0c =\frac{c_B G_0 K_1(R)}{K_0(R)I_1(R)+K_1(R)I_0(R)}  I_0(r),
\]
whose general solution is given by
\[
p_\infty = - \frac{c_B G_0 K_1(R)}{K_0(R)I_1(R)+K_1(R)I_0(R)}  I_0(r)+ a \ln r+b.
\]
The boundedness of $p_\infty$ at $r=0$ implies $a=0$, and the boundary condition $p_{\infty}(\pm R(t))=0$ indicates 
\[
b=\frac{c_B G_0 K_1(R)I_0(R)}{K_0(R)I_1(R)+K_1(R)I_0(R)}  .
\] 
In sum, we get
\[
p_\infty= \left\{
\begin{array}{ll}
 - \frac{c_B G_0 K_1(R)}{K_0(R)I_1(R)+K_1(R)I_0(R)}  I_0(r)+ \frac{c_B G_0 K_1(R)I_0(R)}{K_0(R)I_1(R)+K_1(R)I_0(R)} , & x \in [ -R(t),R(t)]; \\
0, & x \notin   [ -R(t),R(t)] .
\end{array}
\right.
\]
By direct calculation, we find the front propagation speed
\[
-p'_\infty (R(t))=c_B G_0 \frac{K_1(R)I_1(R)}{K_0(R)I_1(R)+K_1(R)I_0(R)} \le c_B G_0 \frac{K_1(R)I_1(R)}{K_1(R)I_0(R)} = c_B G_0 \frac{I_1(R)}{I_0(R)}\,
\]
which implies that the speed in the \emph{in vivo} model is slower than that in the \emph{in vitro} model. And the limiting speed is $\frac{1}{2} c_B G_0$ as $R(t) \rightarrow \infty$. Finally, we write
\begin{equation} \label{eqn:Rtt-vivo}
\partial_t R(t) =c_B G_0 \frac{K_1(R)I_1(R)}{K_0(R)I_1(R)+K_1(R)I_0(R)} \,.
\end{equation}

\section{Numerical method} \label{sec:numeric}
In this section, we discuss the numerical method for the cell density equations (\ref{eq:n}). Our goal is to obtain a numerical approximation to the cell density model with big $\gamma$ such that it can be compared with the analytical solution derived in the last section to the limiting Hele-Shaw flow. Note that a direct simulation of the cell density model can be very challenging due to the high nonlinearity and degeneracy, in which case the space and time steps have to be small enough to overcome the numerical error or instability induced by large $\gamma$. 


Here we adopt the numerical methods for sub-critical Keller-Segel equations proposed in \cite{LiuWangZhou} to the tumor growth models, which is positivity preserving and conservative when $G=0$, so that it can handle the moving transient front nicely with correct growth in total mass. Besides, it uses a semi-implicit discretization in time so that it is free from nonlinear solvers. More specifically, we consider a 2D case in the following without loss of generality. Denote
\[
M=\exp\left(-n^\gamma \right),
\]
then equation \eqref{eq:n} can be formulted as 
\begin{align*}
\frac{\partial}{\partial t} n &=  \nabla \cdot \left[ n M \nabla \frac{1}{M} \right]+n G(c) \\
&=\nabla \cdot \left[ n M \nabla \frac{n}{nM} \right]+n G(c) \,,
\end{align*}
which can be solved by a semi-discrete semi-implicit scheme 
\begin{align}
\frac{n^{k+1}-n^k}{\Delta t} &=\nabla \cdot \left[ n^k M^k \nabla \frac{n^{k+1}}{n^k M^k} \right]+n^{k+1} G(c^{k+1}), \label{eq:numn}\\
- \Delta c^{k+1} & =- \Psi \left(n^k,c^{k+1}\right)\,. \label{eq:numc}
\end{align}
Here the superscript $k$ stands for the numerical solution at $t=t_k=k\Delta t$. Notice that one can solve for $c^{k+1}$ first from \eqref{eq:numc}, and then solve for $n^{k+1}$ from \eqref{eq:numn}, and thus no nonlinear solver is needed as long as $\Psi(n,c)$ is linear in $c$. Clearly, the three models that we have studied satisfies this condition. 

For spatial discretization, we notice that a standard five point discretization of (\ref{eq:numc}) guarantees boundedness of numerical approximations of $c$. That being said, if we denote the fully discrete approximation of $c$ at  $(x_i,y_j, t_k)$  by $c^k_{i,j}$, then we have \[
0\le c^k_{ij} \le c_B.\]
Note that, equation  \eqref{eq:numn} can be reformulated as 
\begin{equation} \label{eq:nre}
\left(1- \Delta t  G(c^{k+1}) \right) n^{k+1}-\nabla \cdot \left[ n^k M^k \nabla \frac{n^{k+1}}{n^k M^k} \right]=n^k.
\end{equation}
Clearly, if $\Delta t$ satisfies the following condition
\begin{equation}
\Delta t < \min_x\{1/G(c^{k+1})\}=\frac{1}{G(c_B)}=\frac{1}{G_m},
\end{equation}
the left hand side of \eqref{eq:nre} is a positive definite operator of $n^{k+1}$. Therefore, as long as the spatial discretization can preserve this property, such as the symmetric framework in \cite{JinYan, JinWang}, the fully discrete numerical scheme is positivity preserving.

In the radial symmetric case, let $r$ be the radius, the system changes to
\begin{align*}
&\partial_t n = \frac{\gamma}{\gamma +1 } \frac{1}{r} \frac{\partial }{\partial r} \left(r \frac{\partial}{\partial r} n^{\gamma+1} \right) +n G(c), \\
&- \frac{1}{r} \frac{\partial }{\partial r} \left(r \frac{\partial}{\partial r} c \right) = -\Psi \left( n,c \right).
\end{align*}
Again, denote
\[
M=\exp\left(-n^\gamma \right),
\]
we can reformulate
\begin{align*}
\partial_t n & =\frac{1}{r} \frac{\partial }{\partial r} \left(r n M \frac{\partial}{\partial r} \frac{1}{M} \right) +n G(c) \\
 & =\frac{1}{r} \frac{\partial }{\partial r} \left(r n M \frac{\partial}{\partial r} \frac{n}{nM} \right) +n G(c). 
\end{align*}
Therefore, the corresponding semi-discrete semi-implicit scheme becomes
\begin{align}
&\frac{n^{k+1}-n^k}{\Delta t}  = \frac{1}{r} \frac{\partial }{\partial r} \left(r n^k M^k \frac{\partial}{\partial r} \frac{n^{k+1}}{n^k M^k} \right) +n^{k+1} G(c^{k+1}), \\
&- \frac{1}{r} \frac{\partial }{\partial r} \left(r \frac{\partial}{\partial r} c^{k+1} \right)  =-\Psi \left( n^k,c^{k+1} \right).
\end{align}
Similar analysis can be applied to the radial symmetric case. The readers can refer to \cite{LiuWangZhou} for a more general discussion.

In the rest of this section, we provide a heuristic explanation of what conditions a scheme for the cell density model should satisfy such that it can capture its front speed correctly for large $\gamma$. We use the Lax-Wendroff type argument. To explain, let us consider the following model problem:
\begin{equation} \label{eq:model}
\partial_t n + \partial_x f[n]=g[n]\,,
\end{equation}
where $n(x,t)$ is the density function,  $x \in \mathbb R$ and $t \ge 0$. The flux function $f$ and the growth factor $g$ may depend on functions of  $n$, nonlocal transform of $n$ and their spacial derivative. A weak form of (\ref{eq:model}) reads
\begin{equation}
\int_ 0^{\infty} dt \int_{\mathbb R} dx  \left( \phi_t n + \phi_x f  + \phi g \right) =0\,, 
\end{equation}
where $\phi$ is a smooth test function in $\mathbb R \times [0,\infty)$ with compact support.

The numerical scheme is represented as
\begin{equation} \label{eq:scheme}
\rho^{k+1}_j=\rho^{k}_j - \frac{\tau}{h} [ F_j(n^k,n^{k+1})-F_{j-1}(n^k,n^{k+1})] + k G_j(n^k,n^{k+1}),
\end{equation}
with $\tau$, $h$ being respectively the time and space steps. To lighten the notations, we denote $F_j^k=F_j(n^k,n^{k+1})$, $G_j^k=G_j(n^k,n^{k+1})$. Multiply \eqref{eq:scheme} by $\phi^k_j := \phi (x_j,t^k)$, and sum over $j \in \mathbb Z$ and $k \in \mathbb N$, and we get 
\[
\sum_{k=0}^{\infty} \sum_{j=-\infty}^{\infty} \phi_j^k (n^{k+1}_j - n^k_j)=- \frac{\tau}{h}  \sum_{k=0}^{\infty} \sum_{j=-\infty}^{\infty}  \phi_j^k (F^{k}_j - F^k_{j-1}) + \tau \sum_{k=0}^{\infty} \sum_{j=-\infty}^{\infty}  \phi_j^k G^k_j.
\]
With summation by parts, we obtain
\begin{equation} \label{eq:disc}
-\sum_{k=0}^{\infty} \sum_{j=-\infty}^{\infty} (\phi_j^k -\phi_j^{k-1})  n^k_j=  \frac{\tau}{h} \sum_{k=0}^{\infty} \sum_{j=-\infty}^{\infty} ( \phi_{j+1}^k- \phi_j^k  )F^{k}_j + \tau \sum_{k=0}^{\infty} \sum_{j=-\infty}^{\infty}  \phi_j^k G^k_j.
\end{equation}

Consider a family of discretization parameter sets $\{\tau_l, h_l\}_{l\in \mathbb N}$. We assume that $\tau_l \rightarrow 0$ and $h_l \rightarrow 0$ as $l \rightarrow \infty$. Denote the piecewise constant reconstruction of the solution by $\tilde{n}_l(x,t)$, and we assume that as $l \rightarrow \infty$, $\tilde{n}_l$ converges to a piecewise smooth function $\widetilde n$. Moreover, we assume the piecewise constant construction of the flux $F_l$ and the growth $G_l$ converge to $f(\widetilde n)$ and $g(\widetilde n)$. Then, \eqref{eq:disc} implies, as $l \rightarrow \infty$,
\[
\int_ 0^{\infty} dt \int_{\mathbb R} dx  \left(- \phi_t \widetilde n - \phi_x f (\widetilde n) + \phi g ( \widetilde n) \right) =0. 
\]
This means, if the numerical solutions converge and the flux functions and growth functions converge consistently, the numerical solutions converge to the weak solution of the model equation. Then by standard argument, if the numerical solution converges to discontinuous solution at $X(t)$, the propagation of the discontinuity is governed by
\[
\dot X(t) = \frac{[f(\widetilde n)]}{[ \widetilde n]},
\]
where $[s]$ denotes the jump of $s$ at the discontinuity. 

It is interesting to apply the above result to a simple 1D case of (\ref{eq:n}), wherein we denote $\tilde{n}_\gamma$ the limit of the numerical approximation in the vanishing mesh size limit. Then sending $\gamma \rightarrow \infty$, we expect that, for a general class of initial conditions, 
\[
\widetilde n_{\gamma} \rightarrow \widetilde n_\infty =\chi_{D_\infty(t)},
\]
where $D_\infty(t)$ is defined in (\ref{eqn:Dinfinity}). 
Without loss of generality, we look at the right endpoint of $D(t)$ and obtain
\[
[n_{\infty}]=-1, \quad [ - n_{\infty} \partial_x p ]= \partial_x p,
\]
where $\partial_x p$ is understood as the sided limit of $\partial_x p$ from the interior of the support. 
Then, we conclude,
\[
\dot X(t) = - \partial_x p (X(t)),
\]
which agrees with the front propagation speed of the Hele-Shaw flow model.

Note the Lax-Wendroff type argument above does not give us the criterion to check convergence, but it implies, the discretization of the density equation from the conservative form \eqref{eq:model} is the key to capture the correct front propagation speed. We shall numerically verify in the next sections that, in various cases, the proposed numerical method gives numerical solutions with accurate moving boundaries.

\section{Numerical examples} \label{sec:num}
In this section, we conduct several numerical experiments to further investigate the behavior of the tumor growth model with various nutrient dependence. 

\subsection{2D radial symmetric case with constant growth}
We first consider the radial symmetric case in 2D. Here $r$ is chosen in $[0,3]$. For different $\gamma$, $\Delta t$ is chosen small enough such that the scheme is stable. Neumann boundary condition is taken at $r=0$ and Dirichlet condition $n(r=3, t) = 0$ is taken at the right boundary $r=3$. We also let the growing factor $G(c)$ to be uniformly one. 

\hspace{-0.65cm}{\bf Example 1: an expanding disk} Here the initial profile in $n$ is taken as 
\begin{equation} \label{eqn:IC00}
n(r,0) = \left\{ \begin{array}{cc} 0.99 & 0\leq r\leq 0.8 \\ 0 & 0.8< r \leq 3 \end{array} \right. \,,
\end{equation}
so that it resembles a characteristic function in the region $0\leq r \leq 0.8$. 
Fig.\ref{fig:example1} on the left displays the comparison of numerical solution with different $\gamma$, where one sees that the numerically obtained $n$ has a closer shape of a characteristic function for bigger $\gamma$, as we expected. Next we compare the numerical solution with the analytical solution adopted from (\ref{eqn:R-ball}). Specifically, given the fact that $n$ remains a characteristic function on the support of $0\leq r \leq R(t)$, one can write the analytical solution as 
\begin{equation} \label{eqn:n000}
n(r,t)= \chi_{0\leq r \leq R(t)}, \quad R(t) = 0.8 e^{t/2}\,.
\end{equation}
The results are collected in Fig.\ref{fig:example1} on the right, where a remarkable agreement on the front propagation speed is observed, despite that the numerical solution is always below 0.99, due to the reason that $\gamma$ is not large enough.  

\begin{figure}[!ht]
\centering
\includegraphics[width = 0.48\textwidth]{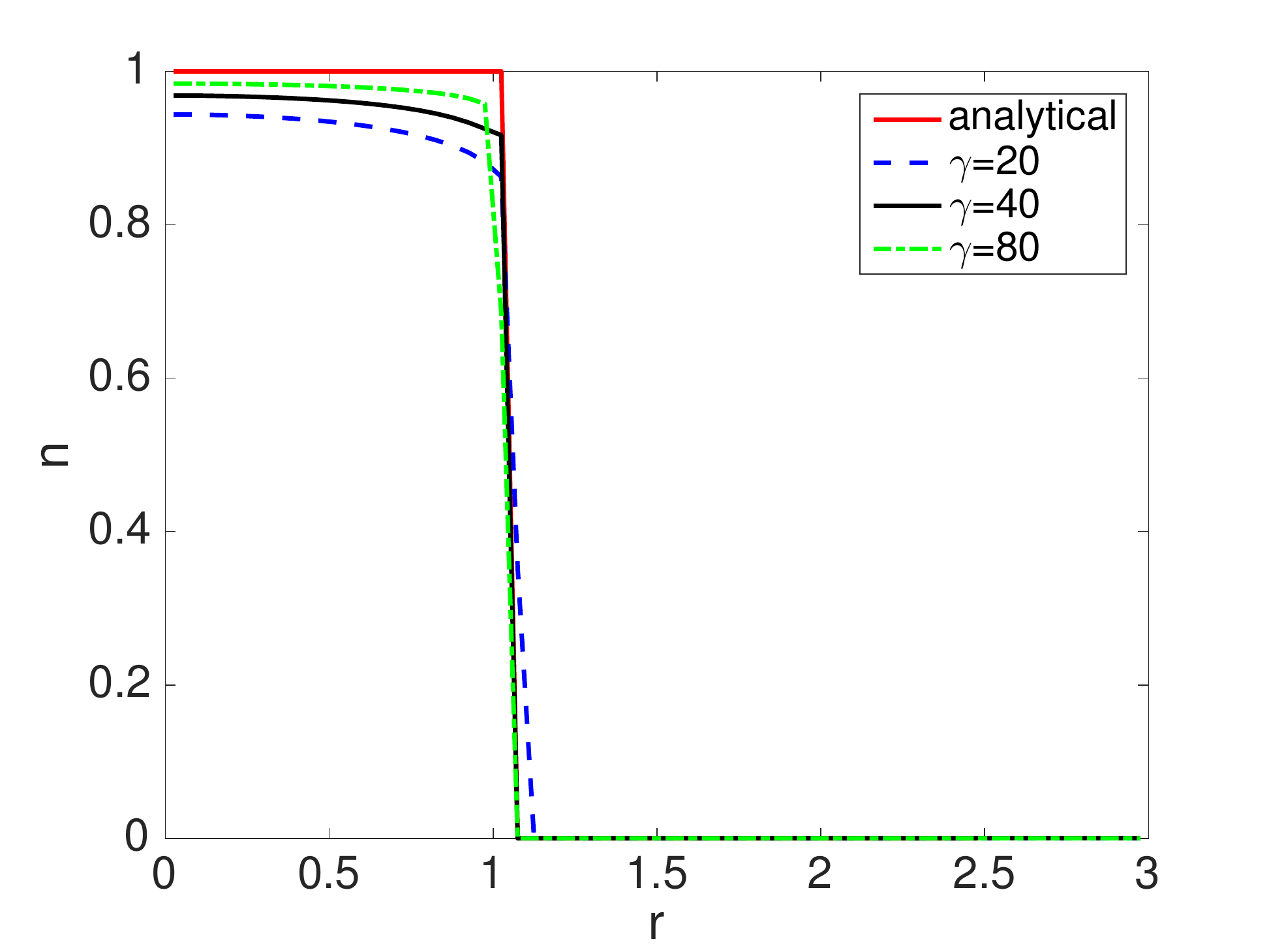}
\includegraphics[width = 0.48\textwidth]{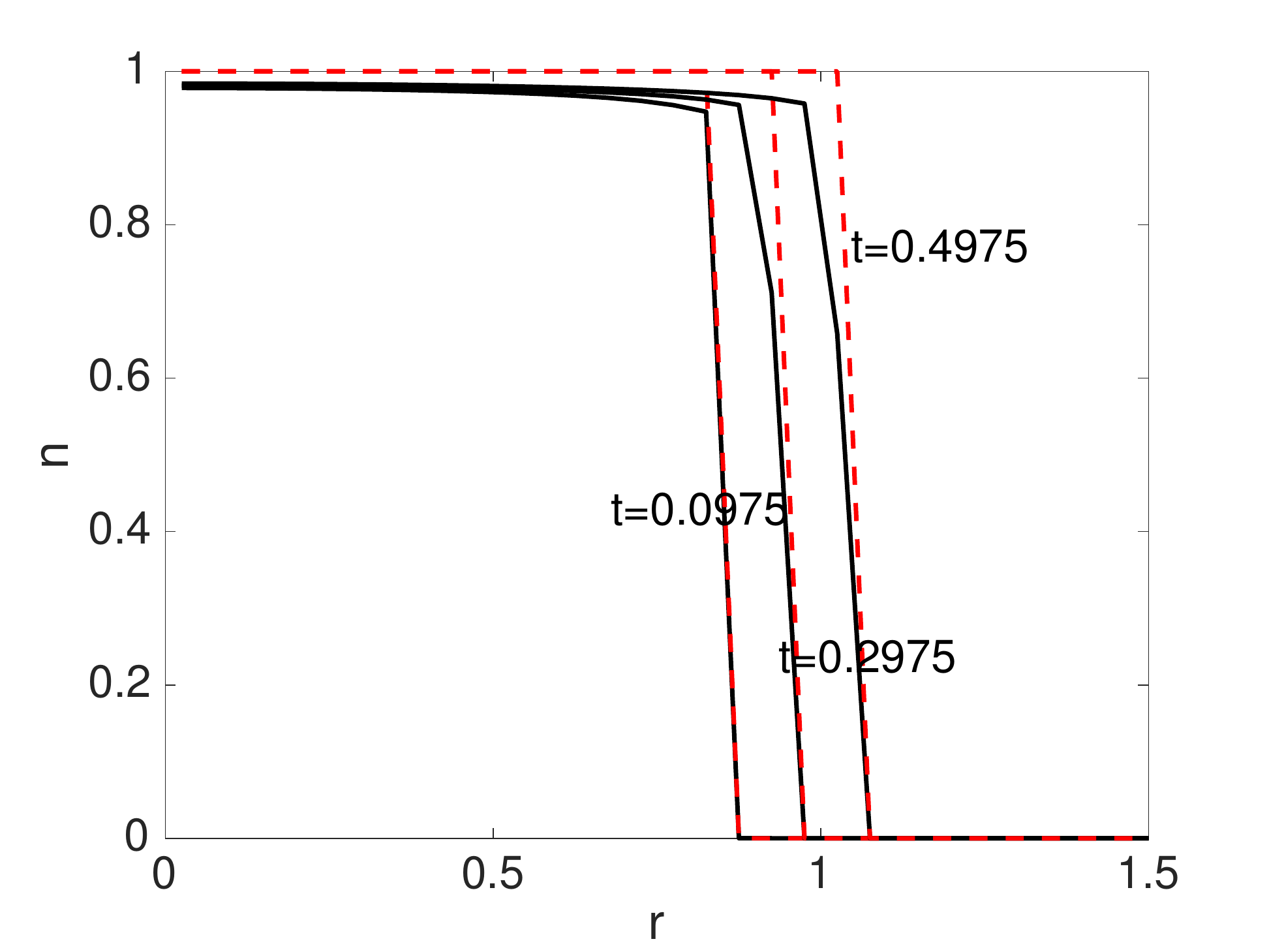}
\caption{Example 1: expanding disk with constant nutrient and initial data (\ref{eqn:IC00}). Left: plot of solution at time $t=0.5$ with different $\gamma = 20,\ 40, \  80$. Here $\Delta r = 0.05$, and $\Delta t = 5e{-5}$ for $\gamma = 20, \ 40$ and $\Delta t = 2.5e{-5}$ for $\gamma=80$. Right: comparison of the numerical solution with $\gamma = 80$ with the analytical solution (\ref{eqn:n000}) at different times $t=0.0975$,  $t=0.2975$, $t=0.4975$. Here the black solid curve is the numerical solution and the red dashed curve is the analytical solution. }
\label{fig:example1}
\end{figure}

\hspace{-0.65cm}{\bf Example 2: a single annulus} In this example, we take initial tumor density to be
\begin{equation} \label{eqn:IC01}
n(r,0) = \left\{ \begin{array}{cc} 0.99 & 0.6\leq r\leq 1 \\ 0 & \text{otherwise} \end{array} \right. \,.
\end{equation}
Then there are two boundaries, one is inside the annulus with initial position $r_-(0)=0.6$, and the other is outside the annulus with initial position $r_+(0)=1$. Again, we conduct two tests, one is with varying gamma, and the other is comparing the numerical solution with the analytical one at different times. The former test produces a result plotted on the left figure in Fig.\ref{fig:example2}. As we expected, when $\gamma$ gets larger, the numerical solution get closer the analytical limiting solution. In the latter test, to get an analytical solution, recall that in Section \ref{sec:analytical solutions}, the boundaries will move according to (\ref{eqn:r-sa}) (\ref{eqn:r+sa}). Thus we numerically solve these coupled ODE system at every time step to get the front position $r_-(t)$ and $r_+(t)$, and recover the analytical solution as
\begin{equation}  \label{eqn:n001}
n(r,t) = \chi_{r_-(t)\leq r \leq r_+(t)}\,.
\end{equation}
Fig. \ref{fig:example2} displays such a comparison at different times, where good agreement of the front speed is observed.
\begin{figure}[!ht]
\centering
\includegraphics[width = 0.48\textwidth]{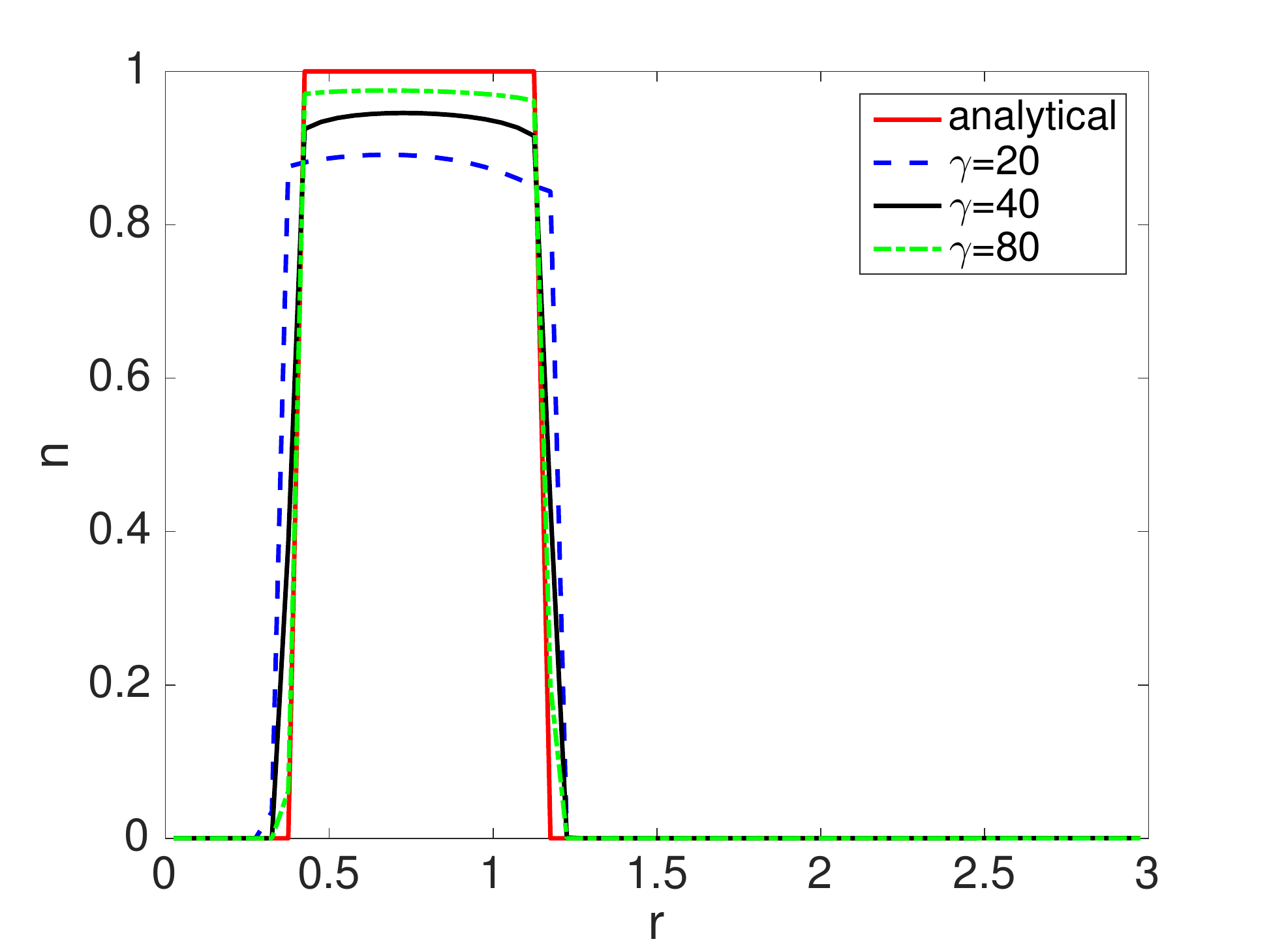}
\includegraphics[width = 0.48\textwidth]{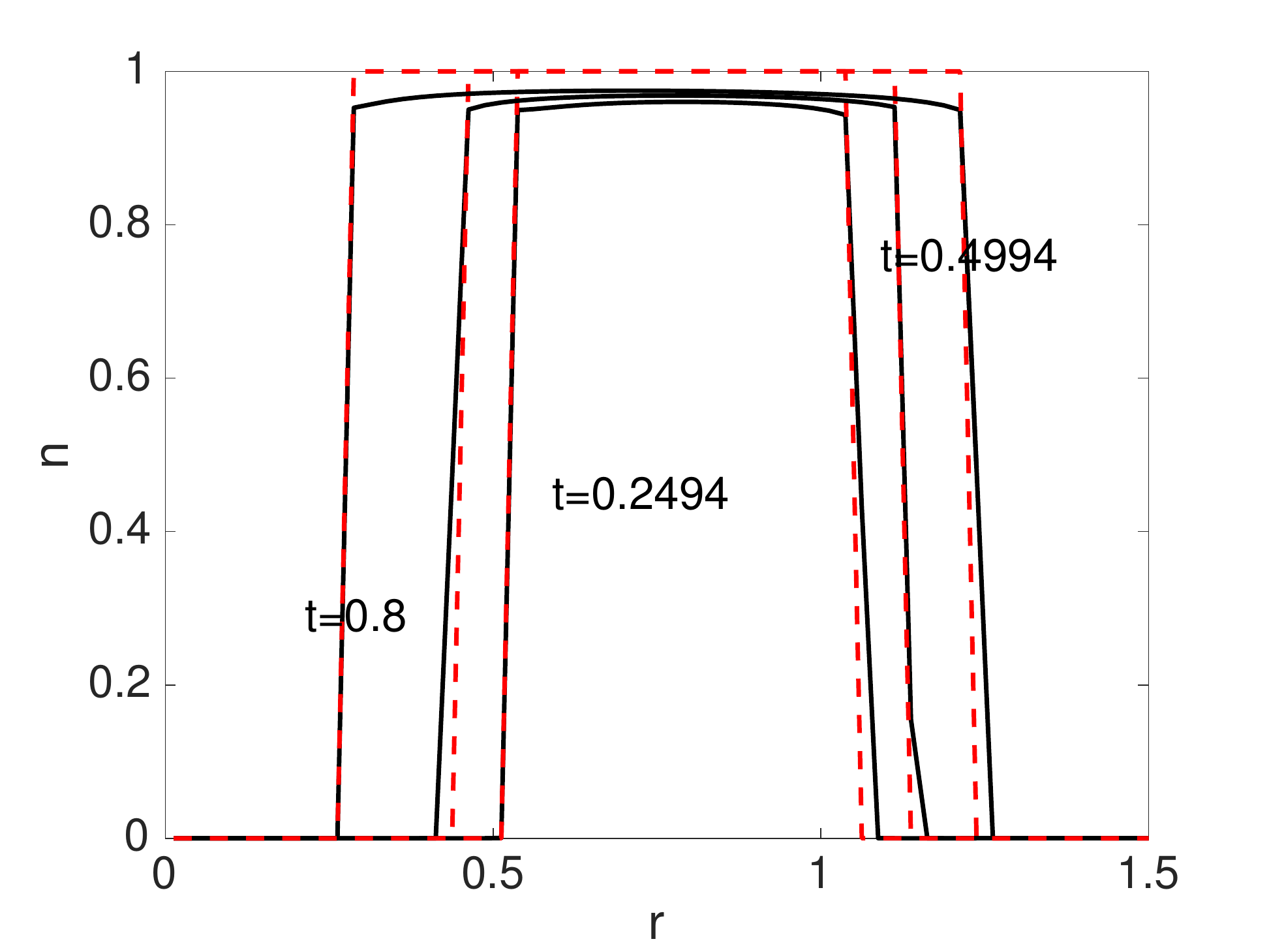}
\caption{Example 2: a single annulus with constant nutrient and initial data (\ref{eqn:IC01}). Left: plot of solution at time $t=0.6$ with different $\gamma = 20,\ 40, \  80$. Here $\Delta r = 0.05$, and $\Delta t = 2.5e{-5}$. Right: comparison of the numerical solution with $\gamma = 80$ with the analytical solution (\ref{eqn:n001}) at different times $t=0.2494$,  $t=0.4994$, $t=0.8$. Here we use $\Delta r = 0.025$ and $\Delta t = 6.25e\!-\!6$. The black solid curve is the numerical solution and the red dashed curve is the analytical solution.}
\label{fig:example2}
\end{figure}

\hspace{-0.65cm}{\bf Example 3: a double annulus} As a direct extension of the second example, we choose initial condition as 
\begin{equation} \label{eqn:IC02}
n(r,0) = \left\{ \begin{array}{cc} 0.99 & 0.6\leq r\leq 0.9  \quad \text{or} \quad 1.5\leq r\leq 1.8  \\ 0 & \text{otherwise} \end{array} \right. \,.
\end{equation}
so that it contains two annulus---the inner one with initial boundaries $r_1(0) = 0.6$, $r_2(0) = 0.9$, and the outer one with initial boundaries $r_3(0) = 1.5$, $r_4(0) =1.8 $. For brevity, we only compare the numerical solution with the analytical solution at different times. The results are given in Fig. \ref{fig:example3}, where the numerical solutions compare favorably with the analytical solution, especially the positions of the boundaries. 

\begin{figure}[!ht]
\centering
\includegraphics[width = 0.48\textwidth]{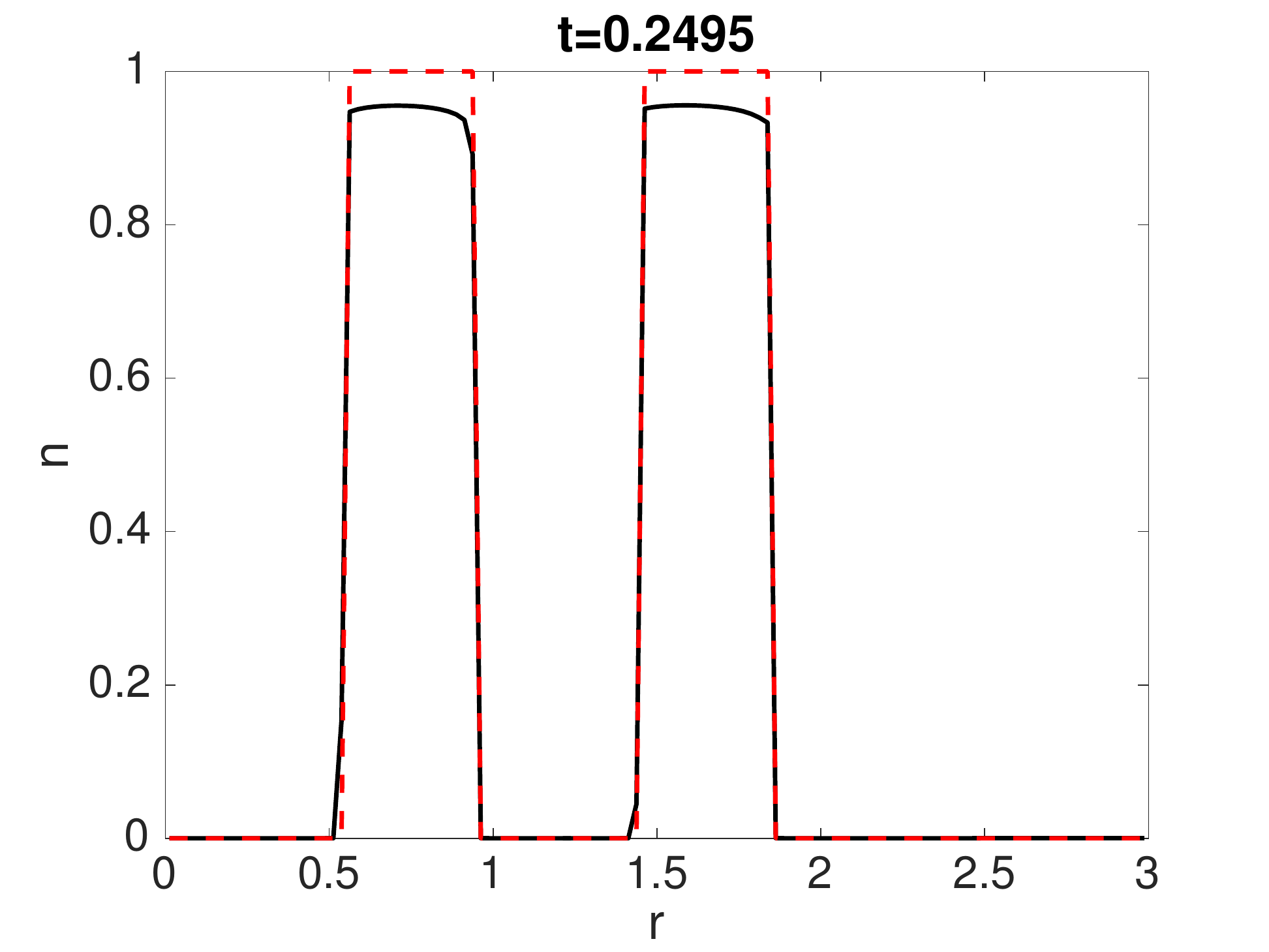}
\includegraphics[width = 0.48\textwidth]{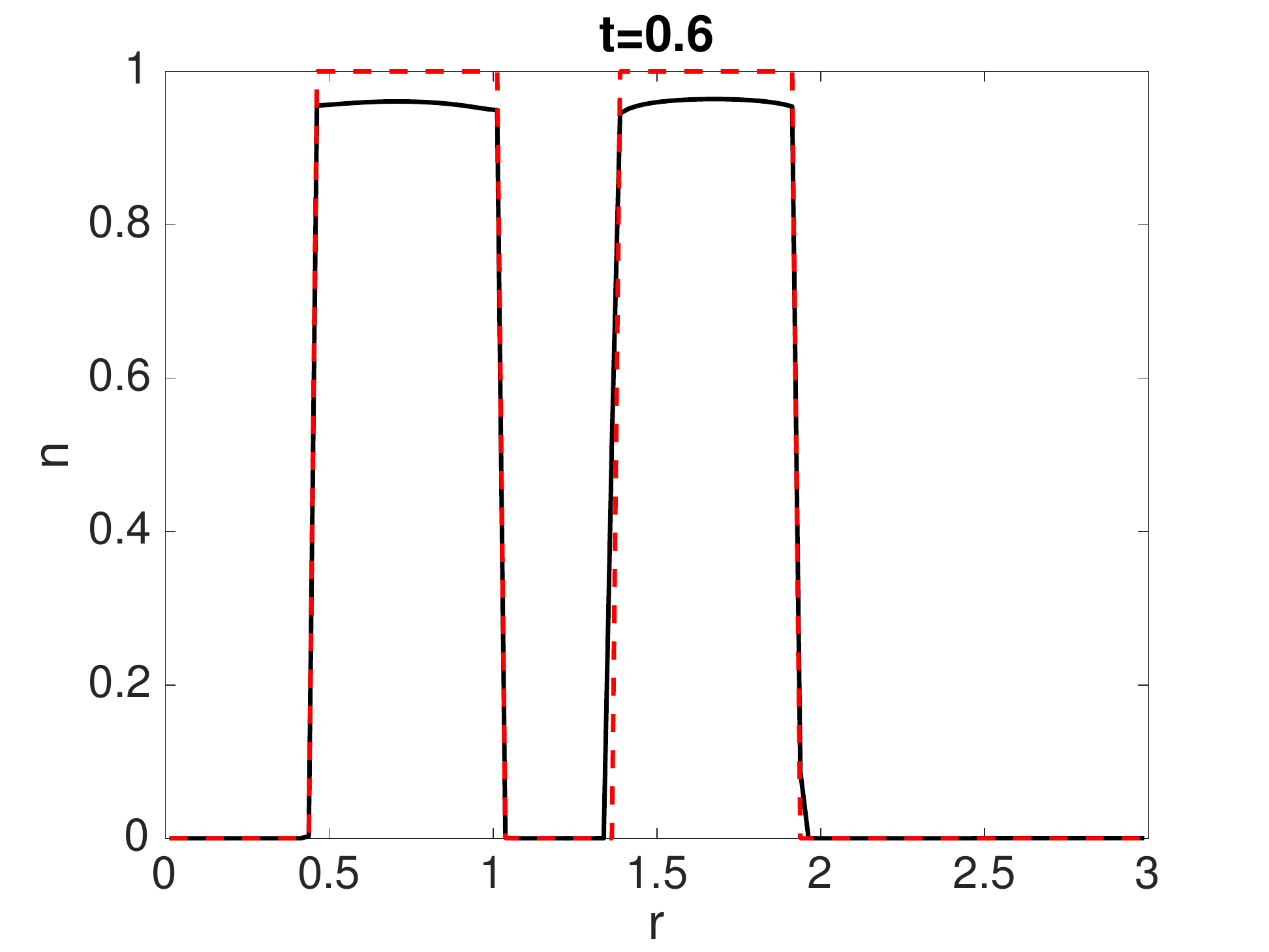}
\caption{Example 3: a double annulus with constant nutrient and initial data (\ref{eqn:IC02}). Here we compare the numerical solution (black solid curve) and analytical solution (red dashed curve) at time $t=0.2495$ (left) and $t=0.6$ (right). Here we use $\Delta r = 0.025$ and $\Delta t = 5e-6$. }
\label{fig:example3}
\end{figure}

\subsection{1D case}
Next, we test the cases when the growing function $G(c)$ has the form (\ref{eqn:G-linear}) with $G_0 = 1$, i.e., $G(c) = c$. Here we only consider the one dimensional setting and let $x\in [-5,5]$. Neumann boundary condition at both ends are used for $n$, whereas Dirichlet boundary condition $c=c_B=1$ are used for $c$ at both ends. The initial condition takes the form
\begin{equation}
n(x,0) = \frac{0.99}{2} (-\tanh(100(x-1)) + \tanh(100(x+1))) 
\end{equation}
such that the two boundaries initially settle at $\pm 1$.

\hspace{-0.65cm}{\bf Example 4: 1D {\it in vitro} model} As always, we test two things here: one is to examine the dependence of the solution on $\gamma$, and the other is to compare the solution with the analytical result. In the left figure of Fig. \ref{fig:example4} , we plot different profiles of $n$ with $\gamma = 20, \ 40, \ 80$, where again as expected, the larger $\gamma$ leads to a shape of $n$ that is closer to the analytical limiting profile. The analytical solution is obtained as
\begin{equation} \label{eqn:n-invtro-000}
n(x,t) = \chi_{-R(t)\leq x \leq R(t)} \,,
\end{equation}
where $R(t)$ is calculated via (\ref{eqn:R-invitro}). Fig. \ref{fig:example4} on the right plots the numerical solution with the analytical one (\ref{eqn:n-invtro-000}) with remarkable agreement. 

\begin{figure}[!ht]
\centering
\includegraphics[width = 0.48\textwidth]{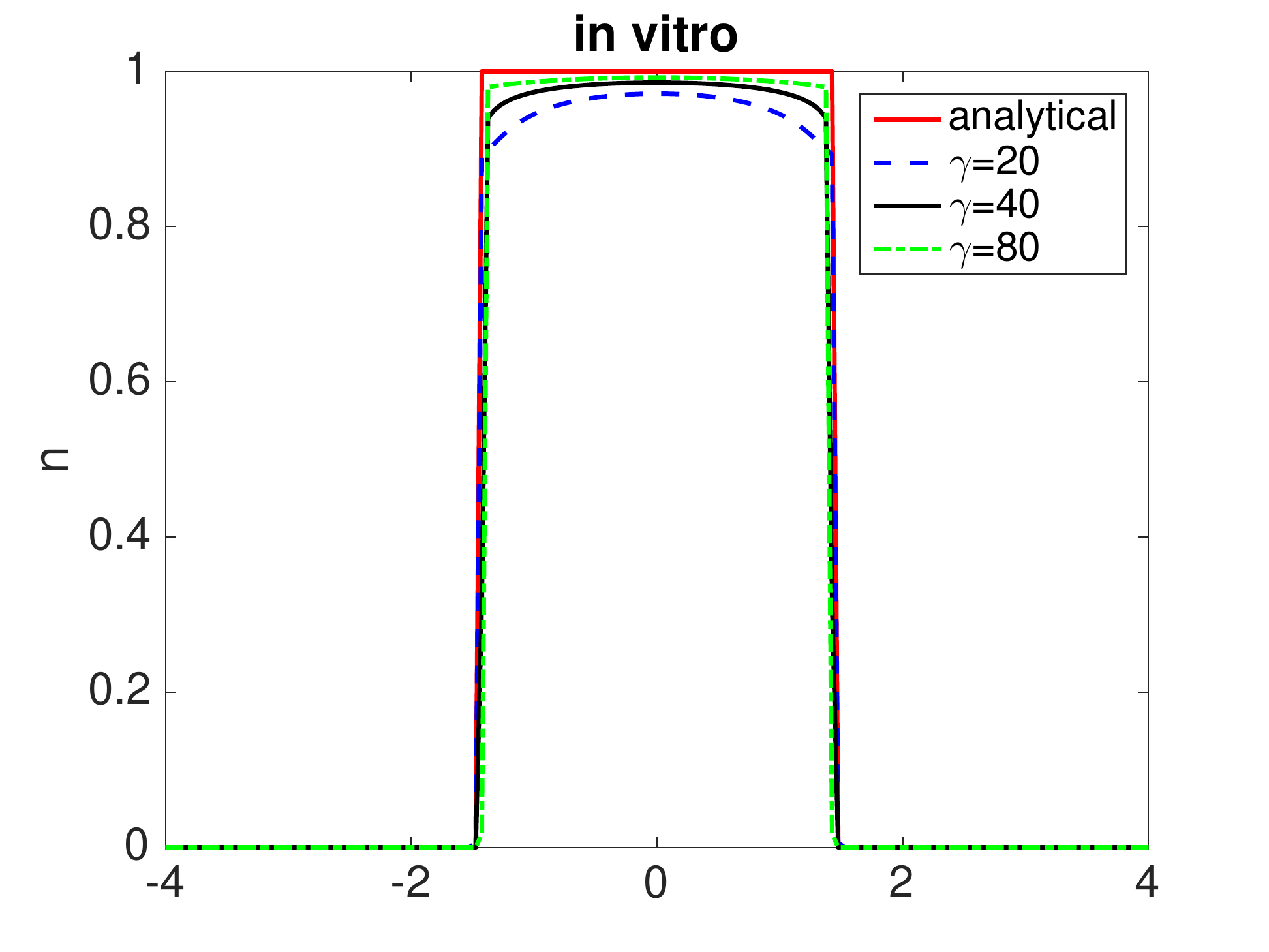}
\includegraphics[width = 0.48\textwidth]{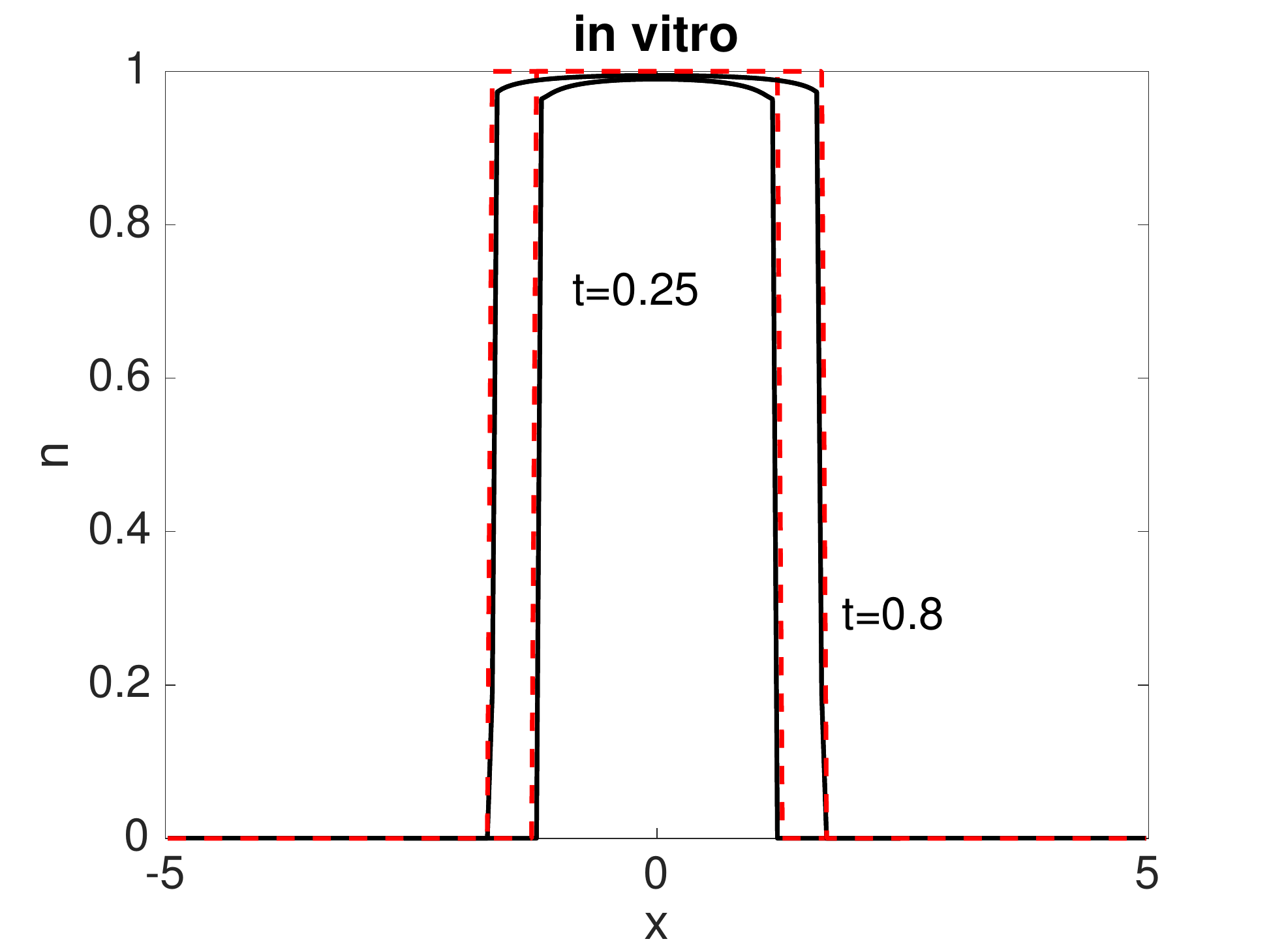}
\caption{Example 4: a 1D {in vitro} model with linear growing function. Left: plots of $n$ at time $t=0.5$ with various $\gamma = 20, \ 40, \ 80$. The red curve is the analytical solution (\ref{eqn:n-invtro-000}). Here $\Delta x = 0.05$ and $\Delta t = 2.5e-5$.}
\label{fig:example4}
\end{figure}

\hspace{-0.65cm}{\bf Example 5: 1D {\it in vitro} model} Similar to the previous example, we generate two plots in Fig. \ref{fig:example5}. Here the analytical solution is take the same form as in (\ref{eqn:n-invtro-000}) but with $R(t)$ obtained by calculating (\ref{eqn:R-invivo-0}) instead. 
\begin{figure}[!ht]
\centering
\includegraphics[width = 0.48\textwidth]{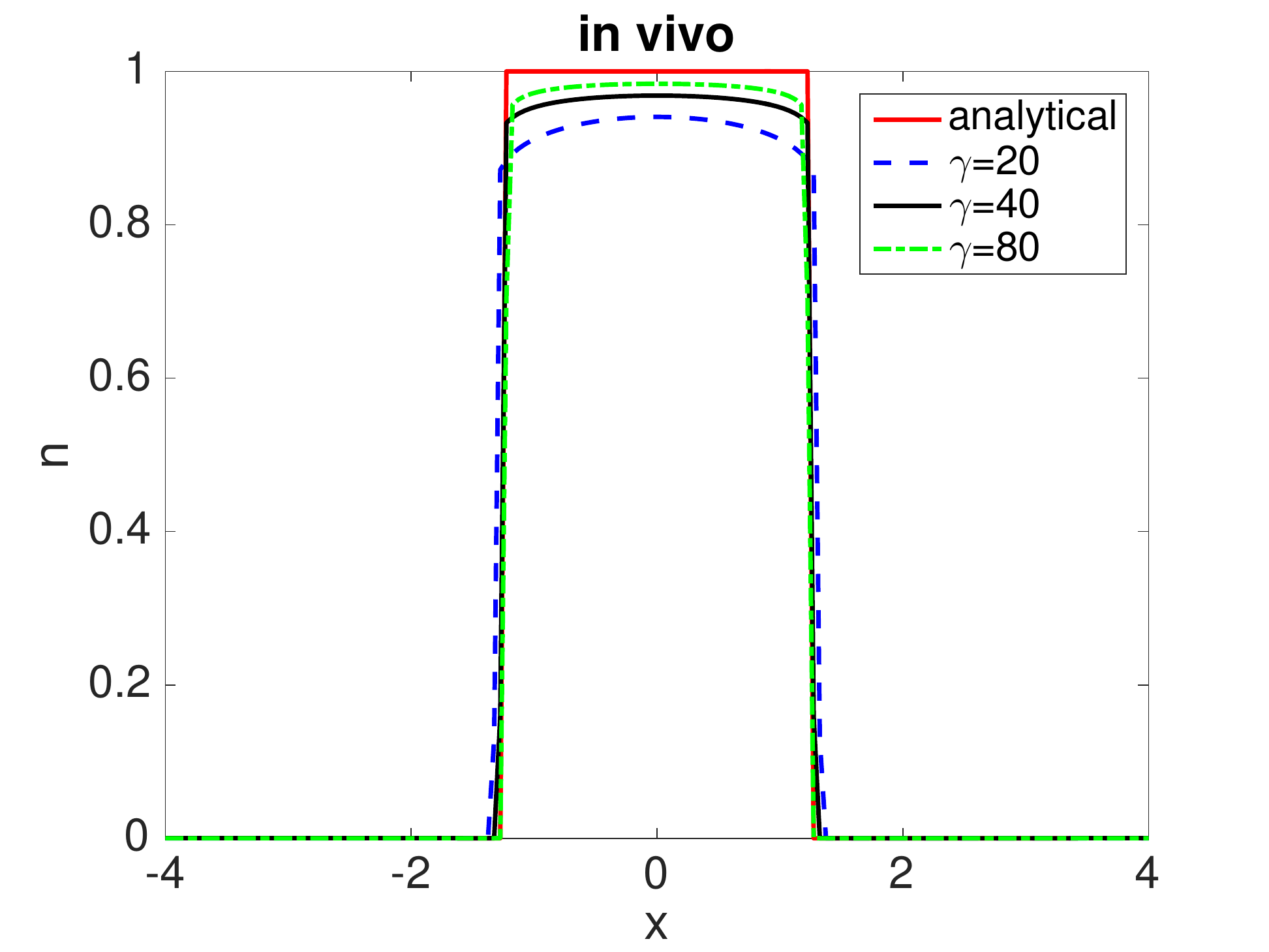}
\includegraphics[width = 0.48\textwidth]{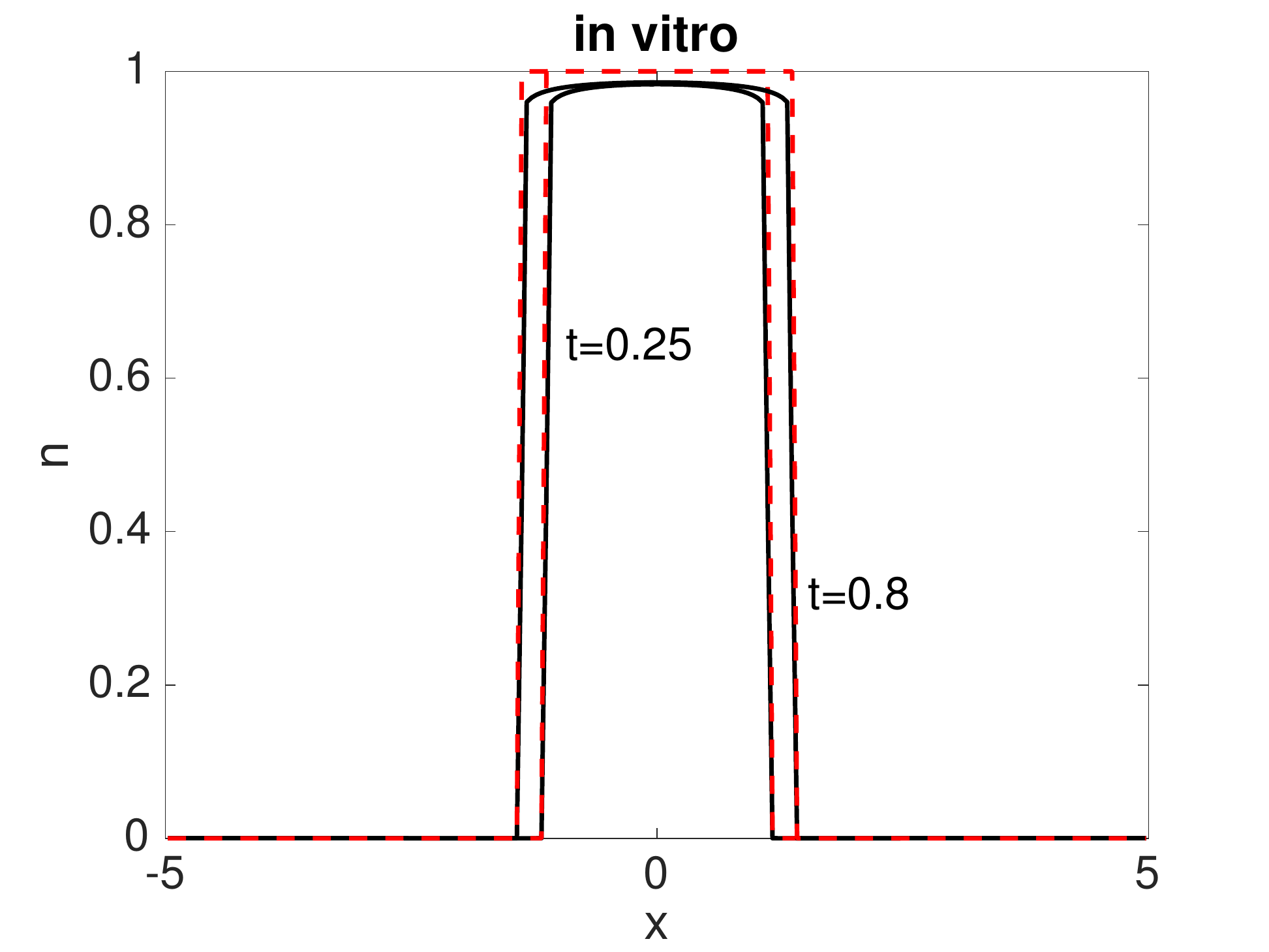}
\caption{Example 5: a 1D {in vivo} model with linear growing function. Left: plots of $n$ at time $t=0.5$ with various $\gamma = 20, \ 40, \ 80$. The red curve is the analytical solution (\ref{eqn:n-invtro-000}). Here $\Delta x = 0.05$ and $\Delta t = 2.5e-5$.}
\label{fig:example5}
\end{figure}
We also compare the front propagation speed of the {\it in vitro} model and {\it in vivo} model. As predicted by (\ref{eqn:R-invitro}) and (\ref{eqn:R-invivo-0}), in the long time limit, the front in the {\it in vitro} model will move twice as fast as that in {\it in vivo} model, and it is confirmed by our Fig. \ref{fig:speed}.
\begin{figure}[!ht]
\centering
\includegraphics[width = 0.6\textwidth]{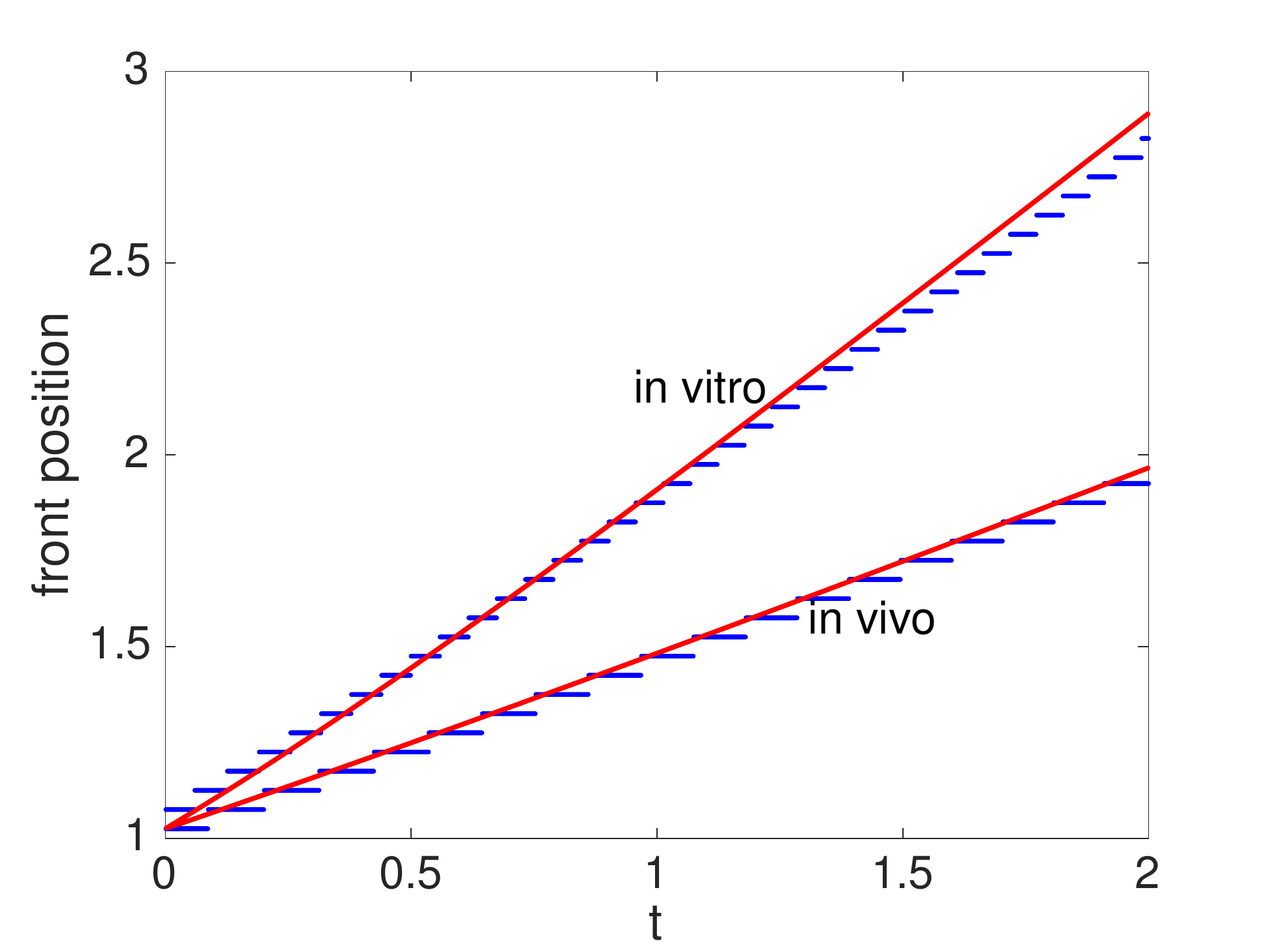}
\caption{A comparison of the front propagation speed for 1D {\it in vitro} model and {\it in vivo} model. The dots represent the the position of the right boundary at each time, and the curves are computed via (\ref{eqn:R-invitro}) and (\ref{eqn:R-invivo-0}). Here $\Delta x =0.05$, $\Delta t = 2.5e-5$, $\gamma = 80$.}
\label{fig:speed}
\end{figure}

\subsection{2D radial symmetric case with linear growth}
{\bf Example $6\&7$: 2D radial symmetric {\it in vitro} and {\it in vivo} model} 
\\Here we again consider linear growth with $G(c) = c$, and evolute $c$ either according to {\it in vitro} or {\it in vivo} model. The initial data is taken as 
\begin{equation} \label{eqn:IC02}
n(r,0) = \left\{ \begin{array}{cc} 0.99 & 0\leq r \leq 0.8    \\ 0 & \text{otherwise} \end{array} \right. \,.
\end{equation}
We choose the computational domain $r \in [0,3]$, and mesh size $\Delta r = 0.05$. Neumann boundary condition is used for both $n$ and $c$ at $r=0$, and Dirichlet boundary condition with $n = 0$ and $c = 1$ are used at $r=3$. For brevity, we only plot the wave front position versus time for these two models with $\gamma = 80$ in Fig.~\ref{fig:2D-vitro-vivo}. Solutions with different $\gamma$ or at different times are very much similar like that in Example 1. As seen in Fig.~\ref{fig:2D-vitro-vivo}, the front propagates at a faster speed in the {\it in vitro} model than the {\it in vivo} model, which is consistent with what we have derived. We also observe a good match between the numerical computed wave front and the analytical ones computed from the limiting model. 
\begin{figure}[!ht]
\centering
\includegraphics[width = 0.6\textwidth]{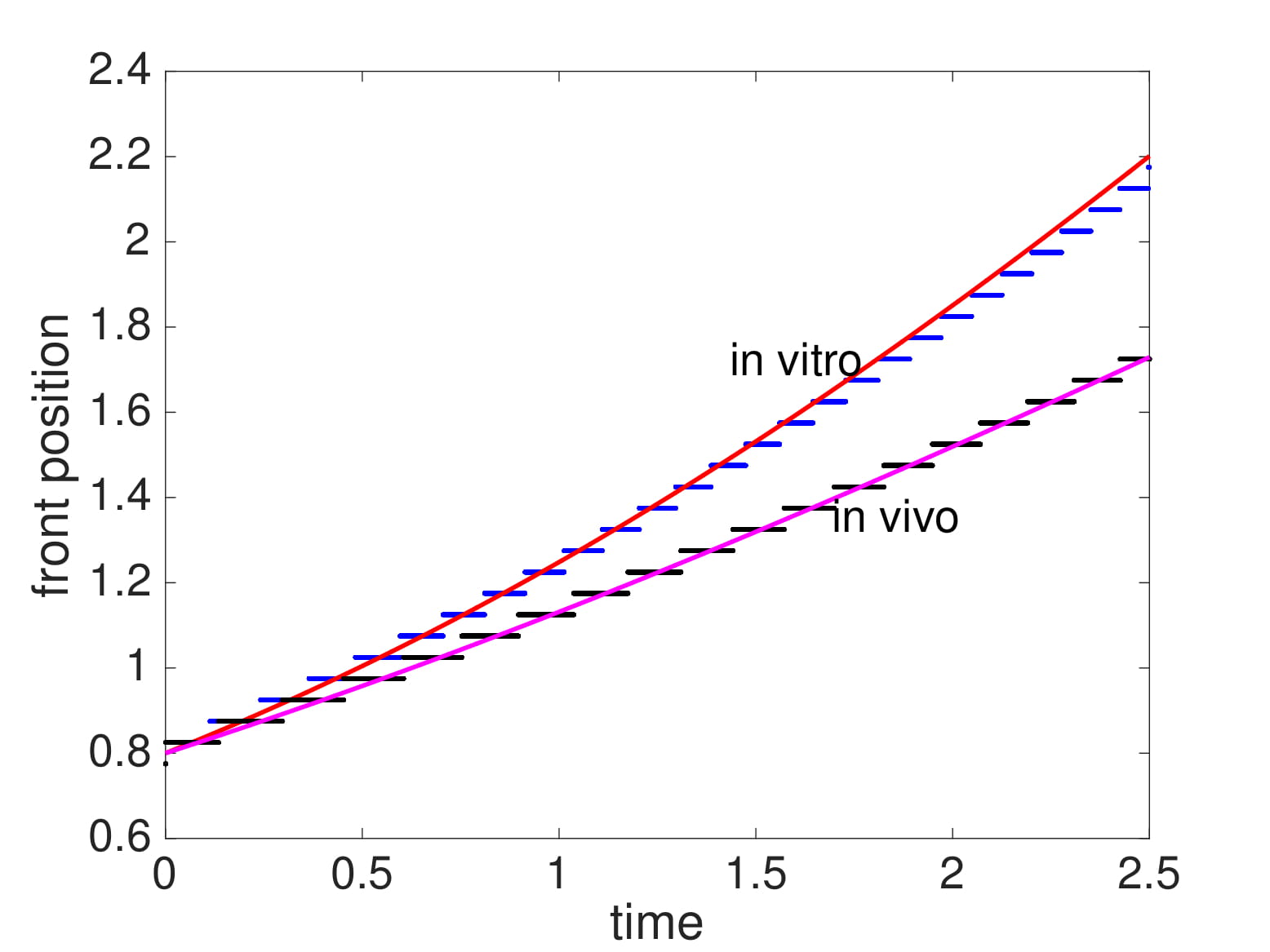}
\caption{A comparison of the front propagation speed in the 2D radial symmetric {\it in vitro} model and {\it in vivo} model. The dots represent the the position of the right boundary at each time, and the curve are computed via (\ref{eqn:Rtt-vitro}) and (\ref{eqn:Rtt-vivo}) . Here $\Delta x =0.05$, $\Delta t = 2.5e-5$, $\gamma = 80$.}
\label{fig:2D-vitro-vivo}
\end{figure}

\subsection{2D geometric motion with constant growth in $c$}
At last, we conduct two 2D examples with constant nutrient, i.e., $G_0= 1$. The computational domain is set to be $(x,y)\in[-2,2]\times[-2,2]$, and $\Delta x  = \Delta y = 2/30$. The first example we compute using the following initial data 
\begin{equation} \label{IC:2D26}
n(x,y,0) = \left\{ \begin{array}{cc} 0.99 & (x,y)\in[0,0.5]\times[0,0.5] ~ \text{or} ~ [-0.6,-0.2]\times[-0.2,0.8]    \\ 0 & \text{otherwise} \end{array} \right. \,.
\end{equation}
In Fig.~\ref{fig:2D-rectangle},  we plot $n$ at different times $t=0,\ 0.0177, \ 0.0311, \ 0.05$, and we see that as times goes, the boundaries of tumors get smeared, and two tumors merge gradually. 
\begin{figure}[!ht]
\centering
\includegraphics[width = 0.45\textwidth]{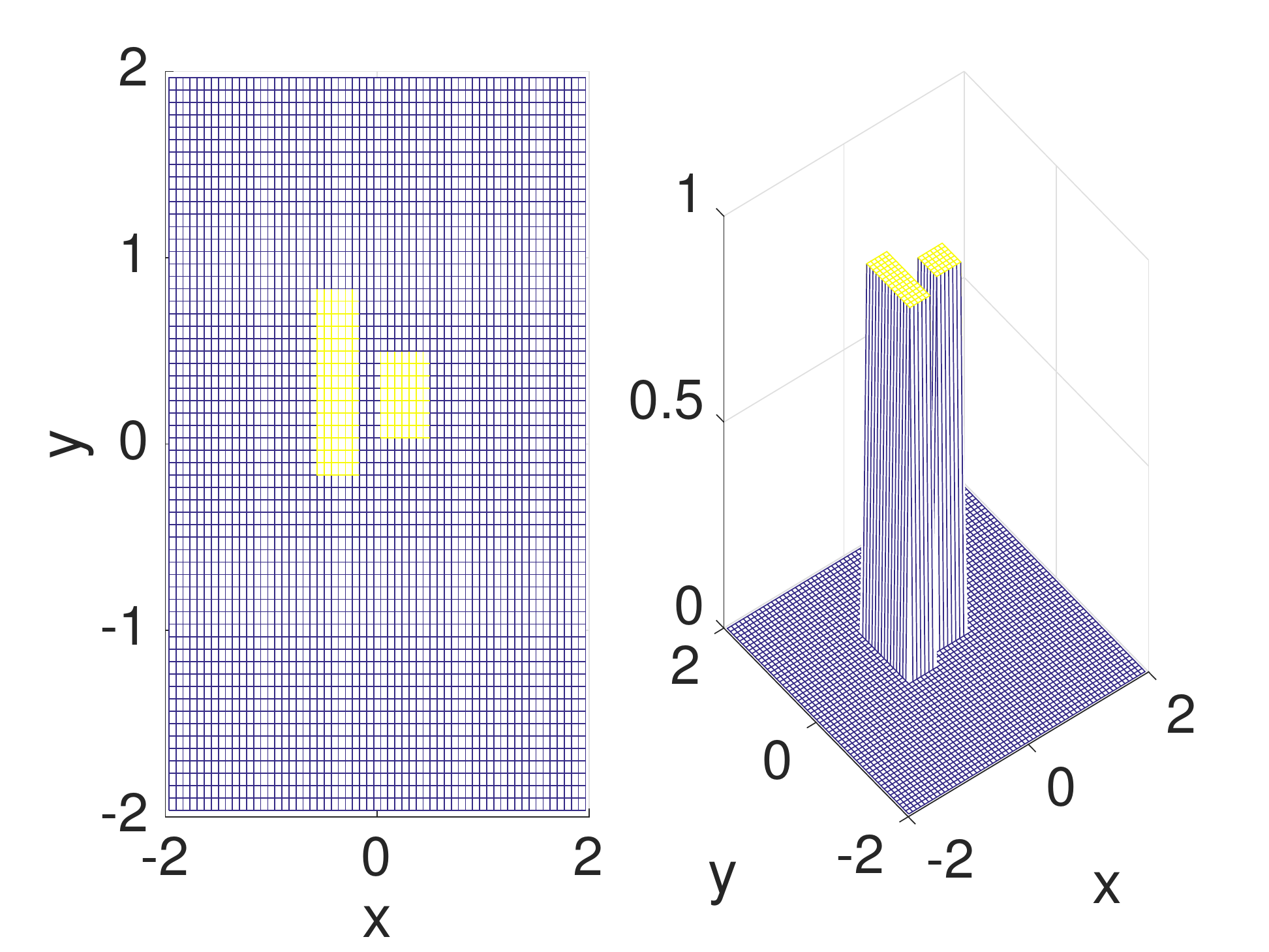}
\includegraphics[width = 0.45\textwidth]{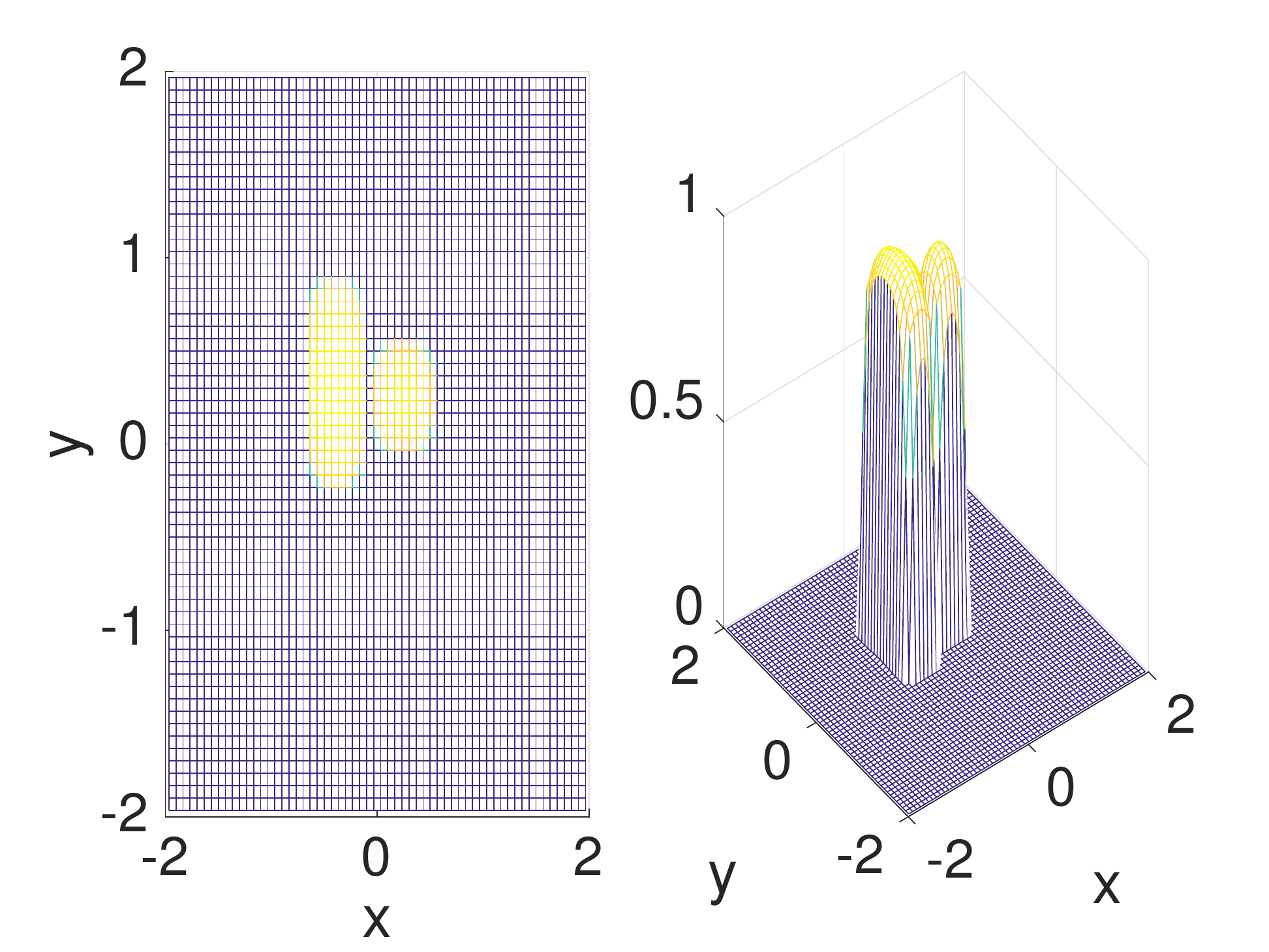}
\\
\includegraphics[width = 0.45\textwidth]{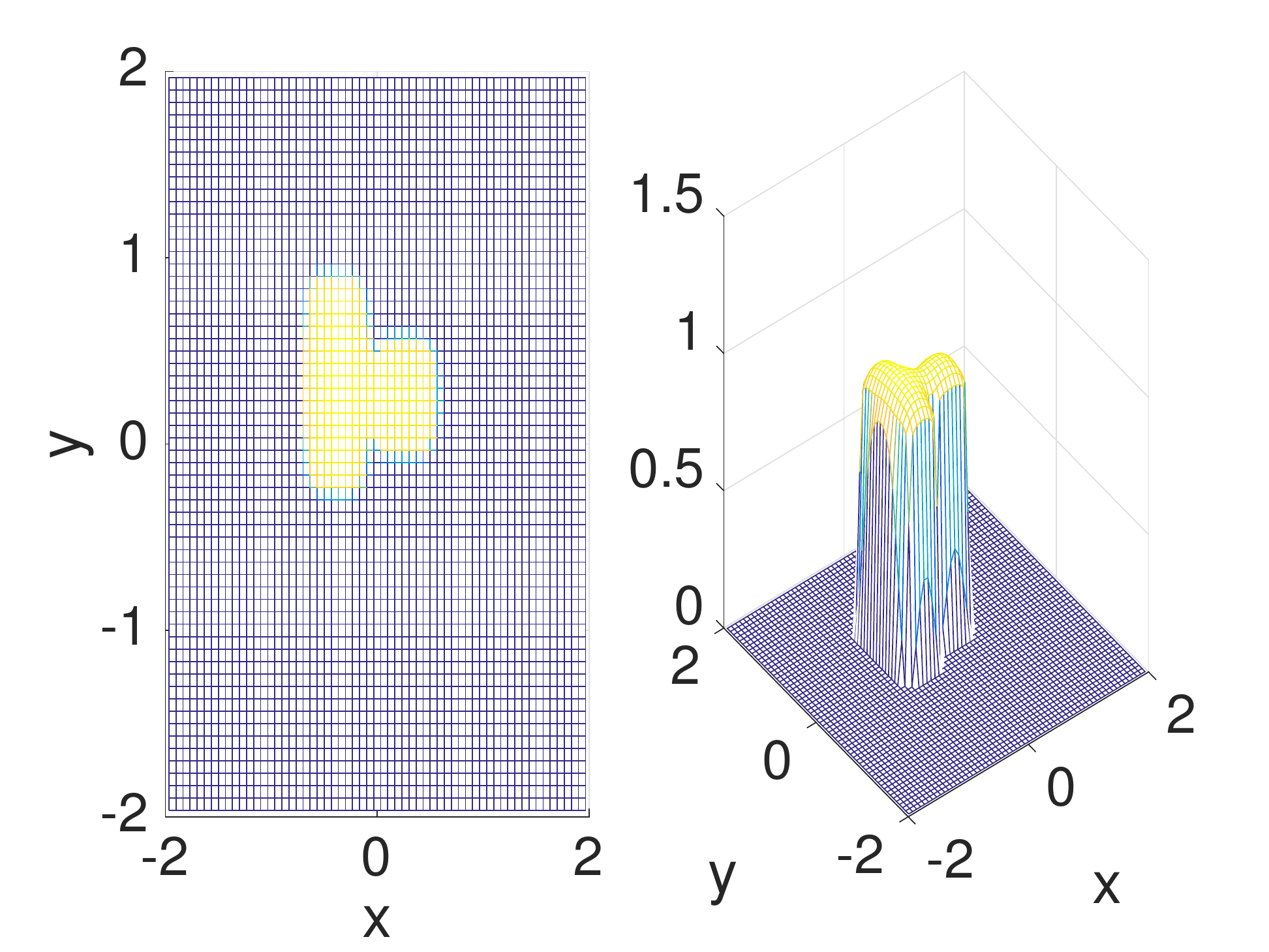}
\includegraphics[width = 0.45\textwidth]{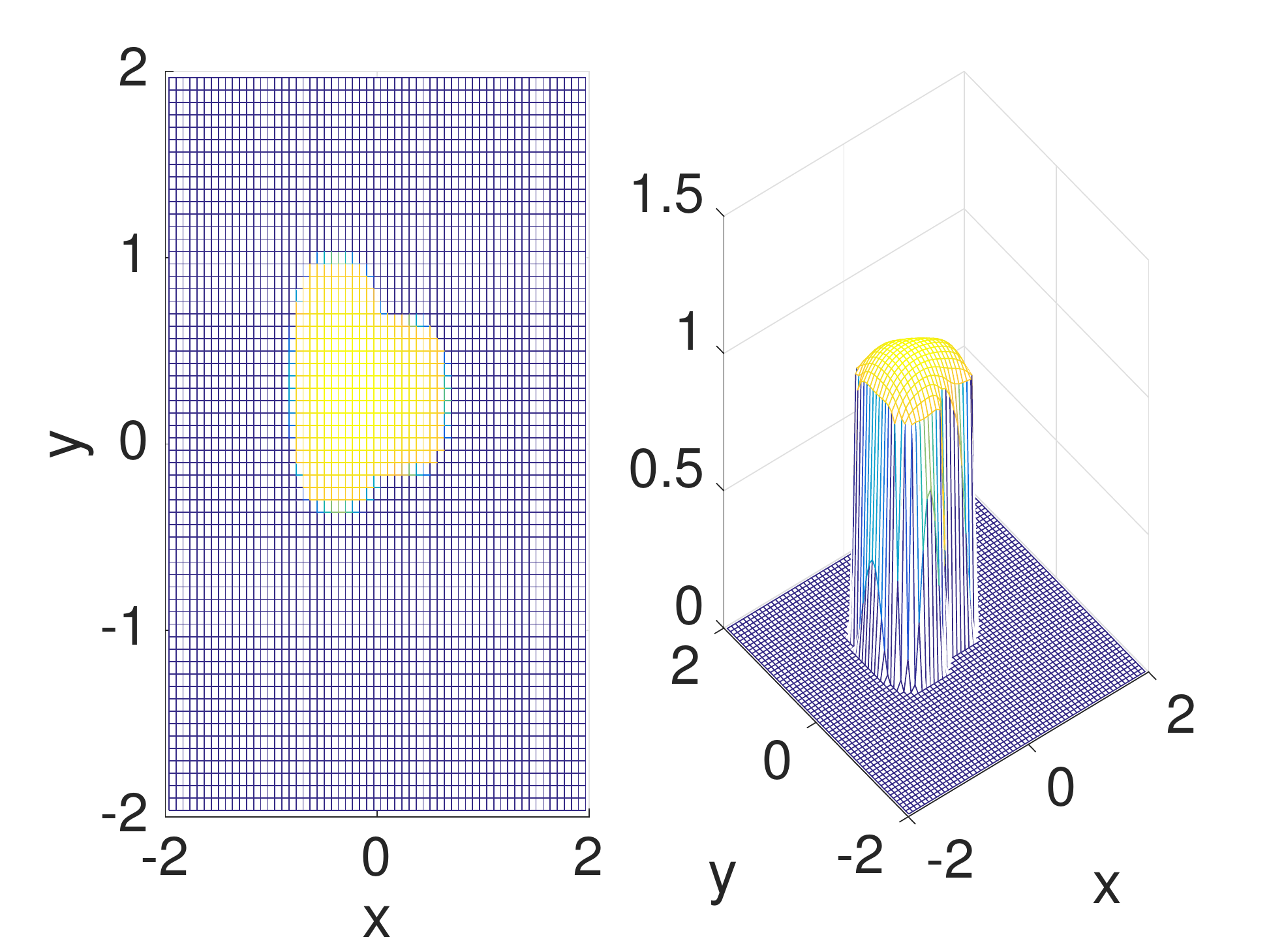}
\caption{Plot of $n$ at four different times with initial data (\ref{IC:2D26}). From left to right, up to down, $t = 0$, $t = 0.0177$, $t=0.0311$, $t = 0.05$.}
\label{fig:2D-rectangle}
\end{figure}

In the second example we use initial data
\begin{equation}\label{IC:2D27}
n(x,y,0) = \left\{ \begin{array}{cc} 0.9 & \sqrt{x^2 + y^2}-0.5-\sin(4\arctan(y/x))/2<0   \\ 0 & \text{otherwise} \end{array} \right. \,.
\end{equation}
and again we plot $n$ at different times . The results are collected in Fig.~\ref{fig:2D-flower}. Here it is important to note that since there exist no upper bound for the pressure and $\gamma$ is not large enough, the maximum density may exceed 1, which induces severe accuracy and stability requirements of the mesh sizes and time steps. Designing more efficient numerical schemes will be our future work. 
\begin{figure}[!ht]
\centering
\includegraphics[width = 0.45\textwidth]{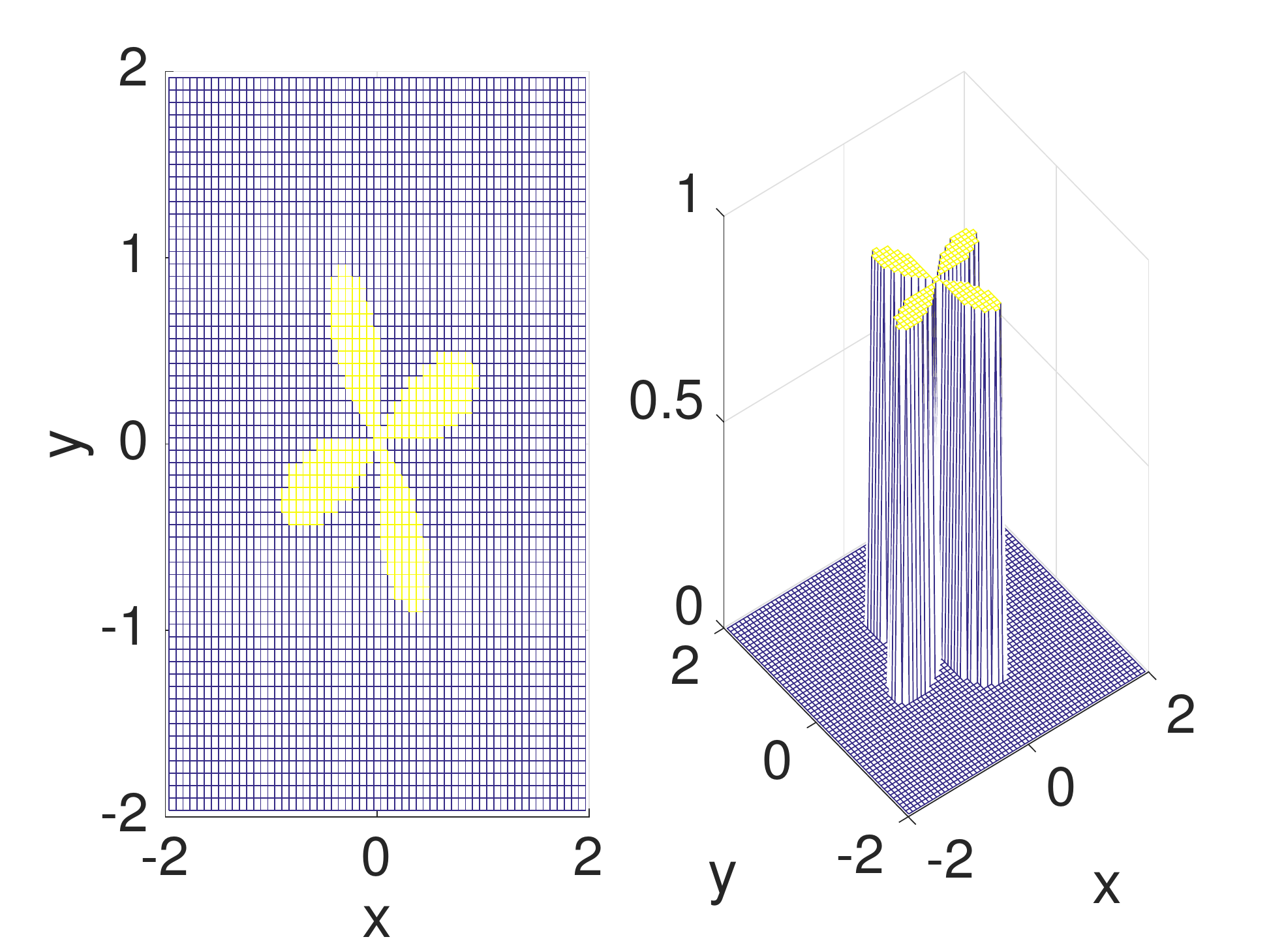}
\includegraphics[width = 0.45\textwidth]{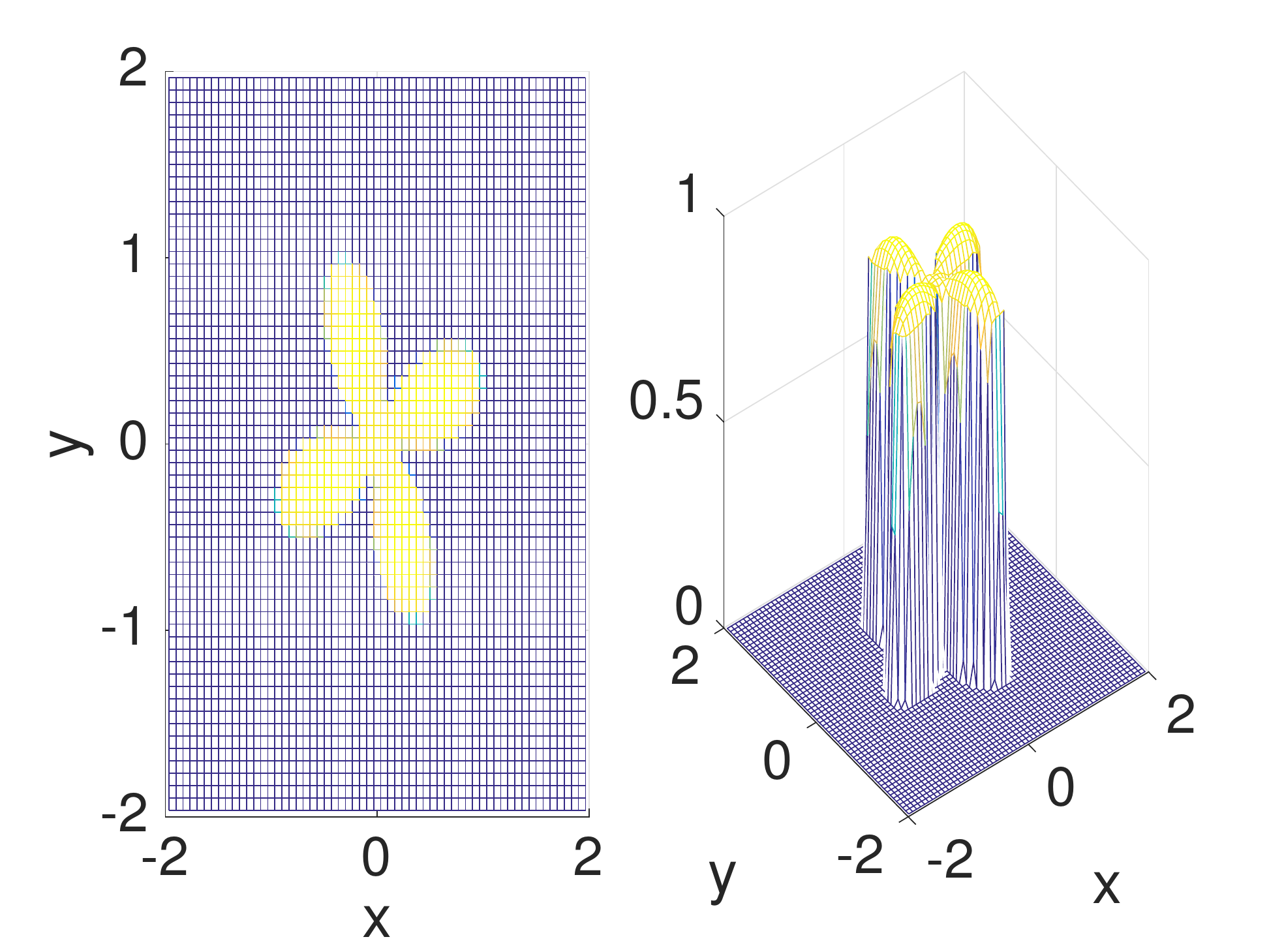}
\\
\includegraphics[width = 0.45\textwidth]{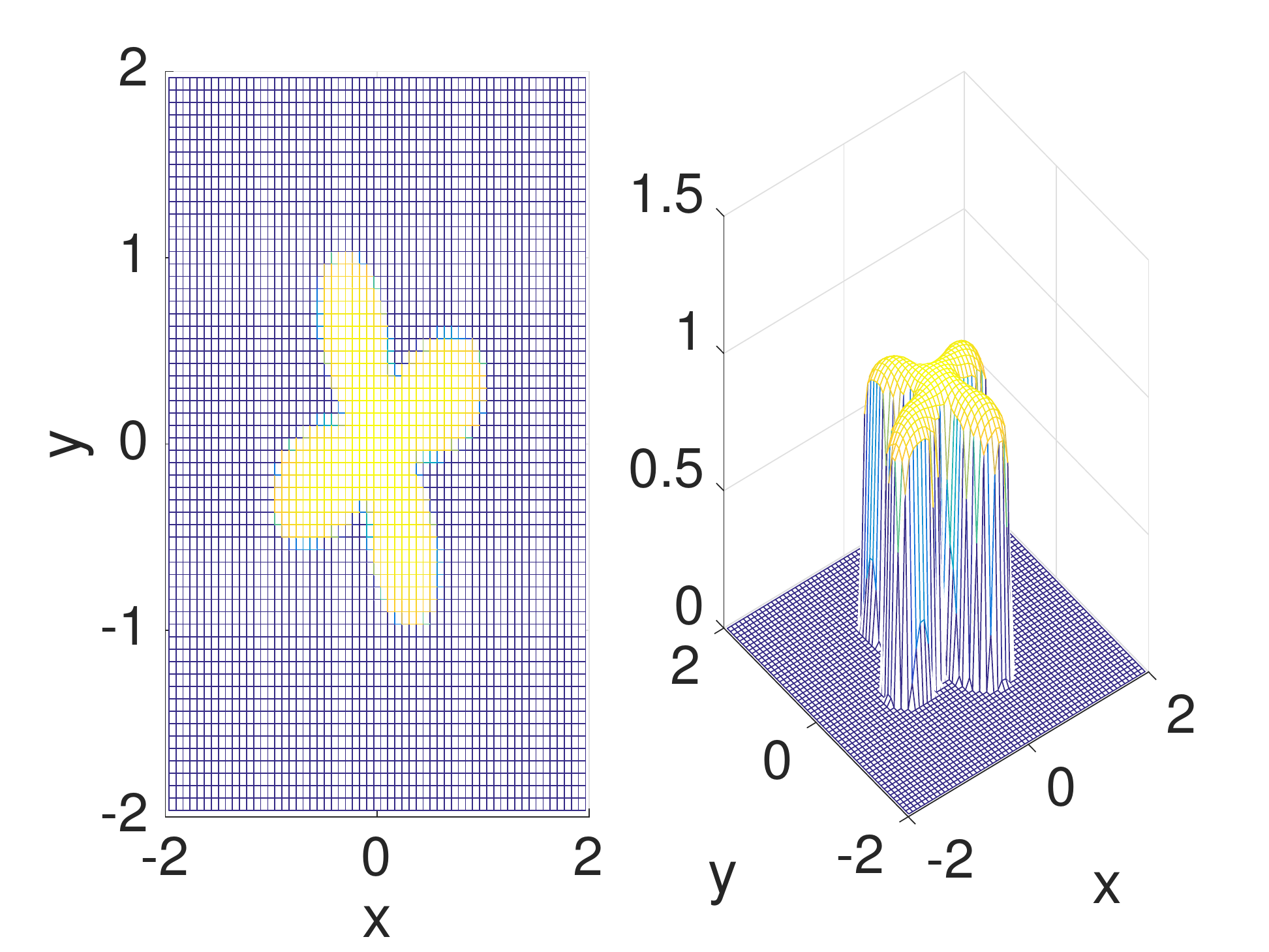}
\includegraphics[width = 0.45\textwidth]{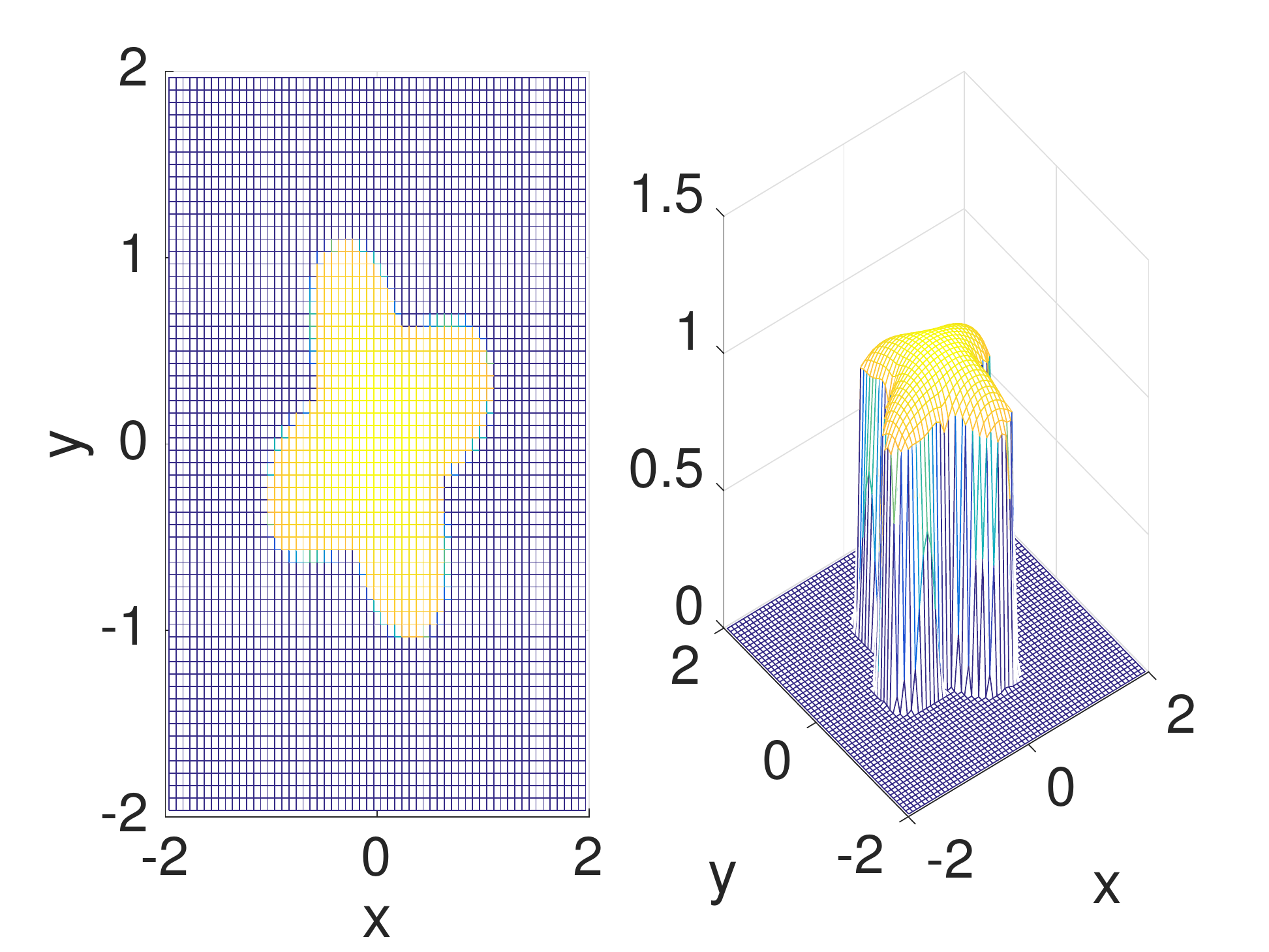}
\caption{Plot of $n$ at four different times with initial data (\ref{IC:2D27}). From left to right, up to down, $t = 0$, $t = 0.0177$, $t=0.0311$, $t = 0.05$.}
\label{fig:2D-flower}
\end{figure}

\section*{Acknowledgments}
J. Liu is partially supported by KI-Net NSF RNMS grant No.11-07444 and NSF grant DMS-1514826.
M. Tang is supported by Science Challenge Project No. TZZT2017-A3-HT003-F and NSFC 91330203.
Z. Zhou is partially supported by RNMS11-07444 (KI-Net) and the start up grant from Peking University. L. Wang is partially supported by the start up grant from SUNY Buffalo and NSF grant
DMS-1620135. M. Tang and L. Wang would like to thank Prof. Jose Carrillo for fruitful discussions on free boundary problems.

\end{document}